%% file: pi_property,uniformly_convex_and_UBCP.tex
\documentclass[10pt,leqno]{amsart}
\usepackage[letterpaper,textwidth=6.2truein,textheight=8.6truein,centering]{geometry}

\usepackage{bm}
\usepackage{bbm}
\usepackage{mathrsfs}
\usepackage{amsmath}
\usepackage{mathtools}
\usepackage{microtype}
\usepackage{amssymb}
\usepackage{amsthm}

\usepackage{cite}
\usepackage{graphicx}
\usepackage{array}
\usepackage{color}
\usepackage[all]{xy}
\usepackage[hidelinks]{hyperref}
\hypersetup{pdftitle={Pi-Properties and Uniform Ball Coverings for Operator Spaces on Separable Uniformly Convex Spaces},
    pdfauthor={Rui Liu and Jie Shen}}

\newtheorem{theorem}{Theorem}[section]
\newtheorem{lemma}[theorem]{Lemma}
\newtheorem{proposition}[theorem]{Proposition}
\newtheorem{corollary}[theorem]{Corollary}

\theoremstyle{definition}
\newtheorem{definition}[theorem]{Definition}
\newtheorem{example}[theorem]{Example}

\theoremstyle{remark}
\newtheorem{remark}[theorem]{Remark}

\numberwithin{equation}{section}

\allowdisplaybreaks
\input{abbreviation}

\begin{document}

\title[\(\pi\)-Properties and Uniform Ball Coverings]
    {\(\pi\)-Properties and Uniform Ball Coverings for Operator Spaces on Separable Uniformly Convex Spaces}

\author{Rui Liu}
\address{School of Mathematical Sciences and LPMC, Nankai University, Tianjin
    300071, P.R. China}
\email{ruiliu@nankai.edu.cn}

\author{Jie Shen}
\address{School of Mathematical Sciences and LPMC, Nankai University, Tianjin
    300071, P.R. China}
\email{1710064@mail.nankai.edu.cn}


\date{\today}

\begin{abstract}
    In this paper, we prove a sufficient condition for the uniform ball-covering property of closed operator subspaces invariant under selected truncations.
    We use uniform convexity of the range space and near-contractive finite-rank projection approximations to obtain quantitative coverings with finite-rank centres whenever the truncated images are separable.
    A weak-star version gives the same covering parameters for quotients by arbitrary weak-star closed subspaces.

    In the completely bounded setting, the covering parameters are independent of the matrix level.
    In particular, the duals of all closed subspaces of \(S_1(\ell_2)\) admit coverings with centre norm \(2\) and radius \(1+\varepsilon\),
    for \(0<\varepsilon<1\), at every matrix level.
    Applications include operator-valued Schur multipliers and, for \(1<p<\infty\),
    prescribed-support Fourier multipliers on countable amenable groups and finitely generated free groups.
    The Fourier coverings have finite-support centres in the given operator or completely bounded norms.
    An endpoint example shows failure of the ball-covering property at \(p=1,\infty\).
\end{abstract}
\maketitle

\section{Introduction}
The ball-covering property (BCP) was introduced by Cheng \cite{Cheng06}.
A Banach space has this property if its unit sphere can be covered by countably many balls which avoid the origin.
Luo and Zheng \cite{LZ21} introduced and studied the strong and uniform ball-covering properties.
The strong ball-covering property requires a uniform bound on the radii;
the uniform ball-covering property (UBCP) also requires a positive uniform distance from the covering balls to the origin.
In this paper, the coverings have the form
\[S_Z\subseteq\bigcup_{j=1}^{\infty}B(aG_j,r),\]
where \(\norm{G_j}=1\) for every \(j\geq1\) and \(0<r<a\).
Every separable Banach space admits such a covering, and the nonseparable space \(\ell_\infty\) also has the UBCP \cite{Cheng06}.
Ball coverings of nonseparable spaces depend on the given norm and are related to smoothness, dentability and the Radon--Nikod\'ym property;
see \cite{GLM19,SC13}.

The study of the BCP has extended from commutative function spaces to noncommutative spaces of operators.
For a locally compact Hausdorff space \(K\),
Liu, Liu, Lu and Zheng \cite{LLLZ22} proved that \(C_0(K)\) has the BCP if and only if \(K\) has a countable \(\pi\)-base,
and that these conditions also imply the UBCP.
They also obtained the UBCP of \(\B(H)\) for separable Hilbert spaces \(H\).
For infinite-dimensional \(H\), Bao, Liu and Shen \cite{BLS25} showed that the Calkin algebra \(\B(H)/\mathcal{K}(H)\) fails the BCP.
Huang, Kudaybergenov and Liu \cite{HKL24} proved that a von Neumann algebra has the BCP if and only if it is atomic and has a separable predual.
Under these conditions it also has the UBCP.
They further proved the UBCP for strongly symmetric spaces with Fatou norms over atomic semifinite von Neumann algebras with separable predual.
These results relate ball coverings to both the algebraic structure and the geometry of the norm.

For general operator spaces, approximation properties of the range provide another approach.
If \(E^*\) is separable, Liu, Liu, Lu and Zheng \cite{LLLZ22} proved that every subspace of \(\B(E,\ell_p)\) containing \(\F(E,\ell_p)\) has the UBCP for \(1<p<\infty\).
Siju and Zheng \cite[Theorem~3.1]{SZ26} treated ranges with unconditional basis constant less than two.
Bao, Liu and Shen \cite{BLS25,BLS27} studied induced operator norms and obtained sufficient conditions based on quantitative approximation estimates.

Two recent papers give further results through the geometry of dual spaces and operators.
In the real setting, Langemets, M{\~o}ttus and Saealle \cite[Theorems~3.6--3.7]{LMS26} obtained sufficient conditions for the UBCP of a dual space from uniformly strongly exposed points of its predual unit ball.
Their Theorem~4.5 gives the UBCP of \(\B(X,Y^*)\) when \(X,Y\) are separable and their unit balls are closed convex hulls of uniformly strongly exposed sets.
In particular, their Corollary~4.7 establishes the UBCP of \(\B(L_p[0,1])\) for every \(1<p<\infty\).

Mishra, Paul, Sain and Sohel \cite{MPSS26} use Birkhoff--James orthogonality and strongly exposed points to study the BCP of real operator spaces.
By their Corollary~3.11, if either \(X\) or \(Y^*\) has the Radon--Nikod\'ym property,
then a subspace of \(\B(X,Y)\) containing \(\F(X,Y)\) has the BCP if and only if both \(X^*\) and \(Y\) have the BCP.
Thus, for real spaces, when \(E^*\) is separable and \(F\) is separable and uniformly convex,
their result already gives the BCP for every operator subspace considered in Corollary~\ref{cor:finite} below.
On this class, our conclusion gives the UBCP with quantitative bounds and finite-rank centres.

In this paper, we obtain uniform ball coverings of closed operator subspaces that are invariant under selected truncations and have separable truncated images.
This includes spaces of multipliers with prescribed support.
We also obtain the same covering parameters for weak-star closed quotients and, in the completely bounded setting, at every matrix level.

We use uniform convexity to obtain a uniform translation estimate for near-contractive projections.
Together with norm detection and separability of the truncated images, this yields the following theorem,
where \(\Lambda_F>1\) is defined in \eqref{eq:lambda-threshold}.
\begin{theorem}\label{thm:introA}
    Let \(F\) be a separable uniformly convex Banach space,
    and let \((P_n)\) be finite-rank projections on \(F\) such that \(P_ny\to y\) for every \(y\in F\) and \(\sup_n\norm{P_n}\leq\lambda\) for some \(1\leq\lambda<\Lambda_F\).
    Let \(E\) be a Banach space and let \(Z\subseteq \B(E,F)\) be a closed linear subspace such that
    \[P_nZ\subseteq Z\quad\text{and}\quad P_nZ\text{ is separable}\quad(n\geq1).\]
    Then there exist \(a>r>0\) and finite-rank operators \(G_j\in S_Z\) such that
    \[S_Z\subseteq\bigcup_{j\geq1}B(aG_j,r).\]
    The parameters depend only on \(\lambda\) and the chosen modulus-of-convexity parameters.
    If \(\lambda=1\), then for every \(0<\delta<1\) the covering can be chosen with \(r=a-\delta\).
\end{theorem}
Here \(P_nZ=\{P_nT:T\in Z\}\).
If \(E^*\) is separable, then the theorem applies to every closed subspace of \(\B(E,F)\) containing \(\F(E,F)\),
since each \(P_nZ\subseteq\B(E,P_nF)\) is separable.

Cheng, Cheng and Shi \cite[Theorem~5.4]{CCS09} obtained infinite-dimensional BCP quotients after arbitrarily small renorming.
Under our weak-star approximation assumptions,
Theorem~\ref{thm:quotient} preserves the uniform covering parameters in every weak-star closed quotient.
In particular, Corollary~\ref{cor:trace-duals} gives coverings of \(S_{M_d(Y^*)}\) with centre norm \(2\) and radius \(1+\varepsilon\) for every closed \(Y\subseteq S_1(\ell_2)\),
\(0<\varepsilon<1\) and \(d\geq1\).
The BCP of \(Y^*\) already follows from its predual's Radon--Nikod\'ym property by \cite[Theorem~15]{SC13};
see also \cite{GLM19}.

For operator-valued Schur multipliers from \(S_p[E]\) to \(S_q[F]\) with coefficients in a fixed cb-norm separable space,
Theorem~\ref{thm:operator-schur} gives coverings with finite-support centres of norm \(3\) and common radius \(2+\varepsilon\),
uniformly over \(1\leq p,q\leq\infty\) and all matrix levels, where \(0<\varepsilon<1\).
Example~\ref{ex:schur-no-ase} treats scalar Schur multipliers on \(X=S_p[c_0(I)]\) for uncountable \(I\).
The underlying real space \(X\) is nonseparable and fails the BCP,
so \cite[Theorem~4.5]{LMS26} and \cite[Theorem~3.5]{MPSS26} do not apply to the pair \((X,X)\).

For \(1<p<\infty\), Theorems~\ref{thm:fourier} and~\ref{thm:free-fourier} give coverings with finite-support centres for prescribed-support Fourier multiplier spaces on countable amenable groups and finitely generated free groups.
For every \(0<\delta<1\), the centre norm \(a=a(p,\delta)\) and radius \(a-\delta\) are independent of the group and,
in the completely bounded case, the matrix level.
Corollary~\ref{cor:fourier-quotients} gives the same parameters for every weak-star closed quotient,
with finite-support representatives of the covering directions.
For completely bounded Fourier multiplier spaces,
an example on \((\bigoplus_{n\geq1}\mathbb{Z}_2)\times S_3\) shows failure of the BCP at \(p=1,\infty\).

The paper is organized as follows.
Section~\ref{sec:projection} establishes the projection estimate and the covering lemma.
Section~\ref{sec:operators} applies them to operator subspaces and weak-star closed quotients.
Section~\ref{sec:classical} treats classical, vector-valued and local \(L_p\)-spaces.
Section~5 develops the completely bounded version and its applications to duals of trace-class subspaces,
Schur multipliers and Fourier multipliers, including free groups, an endpoint example and compatible restrictions on locally finite groups.

The following notation will be used throughout the paper.
\begin{itemize}
    \item \(\B(E,F)\) denotes the space of bounded linear operators from \(E\) to \(F\).
    \item \(\F(E,F)\) denotes the space of finite-rank operators from \(E\) to \(F\).
    \item We denote \(\B(X,X)\) and \(\F(X,X)\) by \(\B(X)\) and \(\F(X)\), respectively.
    \item \(\mathcal{K}(X)\) denotes the space of compact operators on \(X\).
    \item \(B_Z\) and \(S_Z\) denote the closed unit ball and unit sphere of \(Z\), respectively.
    \item \(B(z,r)\) denotes the open ball with centre \(z\) and radius \(r\).
    \item \(I_Z\) denotes the identity operator on \(Z\).
\end{itemize}
The scalar field is \(\mathbb{R}\) or \(\mathbb{C}\) in Sections~2--4, and \(\mathbb{C}\) in Section~5.
Operator subspaces in Sections~2--4 carry their ordinary operator norms.
In Section~5, operator spaces carry their specified matrix norms, and completely bounded norms refer to those structures.
In covering statements involving unit vectors, we assume that the underlying space is nonzero.

\section{Projection estimates and uniform ball coverings}\label{sec:projection}
We separate the argument into two steps.
The first is a geometric estimate for near-contractive projections on uniformly convex spaces.
The second converts a uniform translation estimate, norm detection and separable truncated images into a countable uniform covering.

\begin{definition}
    A Banach space \(Z\) has the ball-covering property (BCP) if there are \(z_j\in Z\) and \(r_j>0\) such that
    \[S_Z\subseteq\bigcup_{j=1}^{\infty}B(z_j,r_j)\quad\text{and}\quad r_j<\norm{z_j}\quad(j\geq1).\]
    It has the uniform ball-covering property (UBCP) if there exist \(r>0\) and \((z_j)\subseteq Z\) such that
    \[S_Z\subseteq\bigcup_{j=1}^{\infty}B(z_j,r)\quad\text{and}\quad r<\inf_{j\geq1}\norm{z_j}.\]
\end{definition}
For a uniformly convex Banach space \(F\), define
\[\delta_F(\varepsilon)=\inf\left\{1-\left\|\frac{x+y}{2}\right\|: x,y\in B_F,\ \norm{x-y}\geq\varepsilon\right\}\quad\text{for } 0<\varepsilon\leq2.\]
Define
\begin{equation}\label{eq:lambda-threshold}
    \Lambda_F=1+\frac{1}{8}\sup_{0<\eta<1}(1-\eta)\delta_F(\eta)>1.
\end{equation}
For every \(1\leq\lambda<\Lambda_F\), we can choose \(0<\eta<1\) and \(t>0\) such that
\begin{equation}\label{eq:admissible-parameters}
    0<t<\frac{\delta_F(\eta)}{4}\quad\text{and}\quad \lambda-1<\frac{t(1-\eta)}{2}.
\end{equation}
We define a triple \((\eta,t,\lambda)\) to be admissible if it satisfies these inequalities.

\begin{definition}
    Let \(\lambda\geq1\).
    A Banach space \(F\) has the \(\pi_\lambda\)-property if, for every finite set \(A\subseteq F\) and every \(\varepsilon>0\),
    there is a finite-rank projection \(P:F\to F\) such that
    \[\norm{P}\leq\lambda\quad\text{and}\quad \norm{Py-y}<\varepsilon\quad(y\in A).\]
    For \(\lambda=1\), this defines the metric \(\pi\)-property.
\end{definition}
If \(F\) is separable, this is equivalent to the existence of finite-rank projections \((P_n)\) with \(\sup_n\norm{P_n}\leq\lambda\) and \(P_ny\to y\) for every \(y\in F\).

\begin{lemma}\label{lem:projection}
    Let \(F\) be uniformly convex and let \((\eta,t,\lambda)\) be admissible.
    Let
    \[\mu=\lambda-t(1-\eta).\]
    Then \(0<\mu<1\), and every projection \(P:F\to F\) with \(\norm{P}\leq\lambda\) satisfies
    \begin{equation}\label{eq:projection-estimate}
        \norm{P-tI_F}\leq\mu.
    \end{equation}
    The same conclusion holds when admissibility is defined using any positive lower bound for \(\delta_F(\eta)\) in \eqref{eq:admissible-parameters}.
\end{lemma}
\begin{proof}
    Let \(\Delta=\delta_F(\eta)\) and \(h=\lambda-1\).
    By \eqref{eq:admissible-parameters}, we have
    \[h<\frac{t(1-\eta)}{2}<\frac{\Delta}{4}.\]
    We fix \(x\in S_F\).
    If \(\norm{Px}\leq1-\Delta/2\), then
    \[\norm{(P-tI_F)x}\leq\norm{Px}+t\leq1-\Delta/2+t<1-t\leq\mu.\]
    Suppose now that \(\norm{Px}>1-\Delta/2\).
    Since \(P((x+Px)/2)=Px\), we have
    \[\left\|\frac{x+Px}{2}\right\|\geq\frac{\norm{Px}}{\norm{P}} >\frac{1-\Delta/2}{\lambda}.\]
    Since \(\norm{Px-Px/\lambda}\leq h\), we obtain
    \begin{align*}
        \left\|\frac{x+Px/\lambda}{2}\right\| &>\frac{1-\Delta/2}{\lambda}-\frac{h}{2}\\
        &\geq1-\frac{\Delta}{2}-\frac{3h}{2}>1-\Delta.
    \end{align*}
    And since \(x,Px/\lambda\in B_F\), we have \(\norm{x-Px/\lambda}<\eta\) by uniform convexity.
    Hence \(\norm{x-Px}<\eta+h\), and we obtain
    \begin{align*}
        \norm{(P-tI_F)x} &=\norm{(1-t)Px-t(x-Px)}\\
        &\leq(1-t)\norm{Px}+t\norm{x-Px}\\
        &<(1-t)\lambda+t(\eta+h) =\lambda-t(1-\eta)\leq\mu.
    \end{align*}
    Combining the two cases, we know \(\norm{P-tI_F}\leq\mu\).
    Moreover, \(0<1-t<\mu<1-t(1-\eta)/2<1\).
    In this argument, we used \(\Delta\) only as a positive lower bound for \(\delta_F(\eta)\).
    Hence the same proof gives the last assertion.
\end{proof}

We next formulate the covering step.
\begin{lemma}\label{lem:cover}
    Let \(Z\) be a Banach space and let \((Q_n)\subseteq\B(Z)\).
    Suppose that there are constants \(M\geq1\), \(0<t<1\) and \(0<\mu<1\) such that
    \begin{equation}\label{eq:covering-estimates}
        \norm{Q_n}\leq M\quad\text{and}\quad \norm{Q_n-tI_Z}\leq\mu\qquad(n\geq1).
    \end{equation}
    Assume in addition that \(Q_nZ\) is separable for every \(n\) and that
    \begin{equation}\label{eq:norm-detection}
        \liminf_{n\to\infty}\norm{Q_nz}\geq\norm{z}\qquad(z\in Z).
    \end{equation}
    For every \(\delta\) with \(0<\delta<(1-\mu)/t\), let \(a=M/t\).
    Then there are \(G_j\in\bigcup_nS_{Q_nZ}\) such that
    \[S_Z\subseteq\bigcup_{j=1}^{\infty}B(aG_j,a-\delta).\]
\end{lemma}
\begin{proof}
    We choose \(\varepsilon>0\) such that \(2\varepsilon<1-\mu-t\delta\),
    and let \(\{G_j:j\geq1\}\) be a countable dense subset of \(\bigcup_nS_{Q_nZ}\).
    For \(z\in S_Z\), we use \eqref{eq:norm-detection} to choose \(n\) such that \(s:=\norm{Q_nz}>1-\varepsilon\).
    Since \(s\leq M\), the number \(a-s/t\) is nonnegative.
    Let \(A=Q_nz/s\).
    Then \(A\in S_{Q_nZ}\), and by \eqref{eq:covering-estimates} we have
    \begin{align*}
        \norm{z-aA} \leq\left\|z-\frac{1}{t}Q_nz\right\| +\left(a-\frac{s}{t}\right)\norm{A}\leq\frac{\mu}{t}+\frac{M-s}{t} <a-\frac{1-\mu-\varepsilon}{t}.
    \end{align*}
    We choose \(j\) with \(\norm{A-G_j}<\varepsilon/M\).
    Then
    \[\norm{z-aG_j}<a-\frac{1-\mu-2\varepsilon}{t}<a-\delta.\]
    Since \(M\geq1\) and \(\delta<(1-\mu)/t<1/t\leq a\), we have \(0<a-\delta<a\).
    Every centre \(aG_j\) has norm \(a\).
\end{proof}

\section{Uniform ball coverings of operator subspaces}\label{sec:operators}
\begin{theorem}\label{thm:operator}
    Let \(F\) be a separable uniformly convex Banach space.
    Suppose that \((P_n)\) is a sequence of finite-rank projections on \(F\) satisfying
    \[\sup_n\norm{P_n}\leq\lambda,\qquad 1\leq\lambda<\Lambda_F,\quad\text{and}\quad P_ny\longrightarrow y\quad(y\in F).\]
    Let \(E\) be a Banach space and let \(Z\subseteq \B(E,F)\) be a closed linear subspace such that
    \[P_nZ\subseteq Z\quad\text{and}\quad P_nZ\text{ is separable}\quad(n\geq1).\]
    Choose an admissible triple \((\eta,t,\lambda)\) and let \(\mu\) be the constant from Lemma~\ref{lem:projection}.
    For every \(\delta\) with \(0<\delta<(1-\mu)/t\), let \(a=\lambda/t\).
    Then there are \(G_j\in S_Z\cap\mathcal{F}(E,F)\) such that
    \[S_Z\subseteq\bigcup_{j=1}^{\infty}B(aG_j,a-\delta).\]
\end{theorem}
\begin{proof}
    For \(T\in Z\), we define \(Q_nT=P_nT\).
    By the invariance assumption, \(Q_n\) is a bounded projection on \(Z\) and \(\norm{Q_n}\leq\lambda\).
    By \eqref{eq:projection-estimate}, we have
    \[\norm{Q_n-tI_Z}=\sup_{T\in B_Z}\norm{(P_n-tI_F)T} \leq\norm{P_n-tI_F}\leq\mu.\]
    We fix \(T\in S_Z\) and \(\varepsilon>0\), and choose \(x\in B_E\) such that \(\norm{Tx}>1-\varepsilon\).
    By pointwise convergence on \(F\), we have
    \[\liminf_{n\to\infty}\norm{Q_nT} \geq\lim_{n\to\infty}\norm{P_nTx}=\norm{Tx}>1-\varepsilon.\]
    Letting \(\varepsilon\downarrow0\), we obtain \eqref{eq:norm-detection}.
    Since each image \(Q_nZ=P_nZ\) is separable, we can apply Lemma~\ref{lem:cover} with \(M=\lambda\).
    Its covering directions belong to \(\bigcup_nS_{P_nZ}\).
    Each such operator takes its values in some finite-dimensional space \(P_nF\), and therefore has finite rank.
\end{proof}
\begin{corollary}\label{cor:metric}
    Under the assumptions of Theorem~\ref{thm:operator} with \(\lambda=1\),
    for every \(0<\delta<1\) there exist \(a>1\) and finite-rank operators \(G_j\in S_Z\) such that
    \[S_Z\subseteq\bigcup_{j=1}^{\infty}B(aG_j,a-\delta).\]
\end{corollary}
We choose \(0<\eta<1-\delta\) and \(0<t<\delta_F(\eta)/4\).
By Lemma~\ref{lem:projection}, we have \(\mu=1-t(1-\eta)\) when \(\lambda=1\).
Since \(\delta<1-\eta=(1-\mu)/t\), we obtain the conclusion from Theorem~\ref{thm:operator} with \(a=1/t\).
The upper limit \(1\) for the covering gap is optimal.
If an open ball \(B(z,r)\) meets \(S_Z\), then the reverse triangle inequality gives \(\norm{z}-r<1\).

Theorem~\ref{thm:introA} follows from Theorem~\ref{thm:operator} and Corollary~\ref{cor:metric}.
The required countability condition has the following sufficient form.
\begin{corollary}\label{cor:dual}
    Under the projection assumptions of Theorem~\ref{thm:operator}, suppose that \(E^*\) is separable.
    Then every closed linear subspace \(Z\subseteq \B(E,F)\) satisfying \(P_nZ\subseteq Z\) has the UBCP,
    with the same covering parameters and finite-rank centres.
\end{corollary}
Indeed, by choosing a basis of \(P_nF\), we obtain
\[P_nZ\subseteq \B(E,P_nF)\cong(E^*)^{\dim P_nF}.\]
Since \(E^*\) is separable, so is \(P_nZ\), and we can apply Theorem~\ref{thm:operator}.
\begin{corollary}\label{cor:finite}
    Let \(F\) be a separable uniformly convex Banach space with the \(\pi_\lambda\)-property, where \(1\leq\lambda<\Lambda_F\).
    If \(E^*\) is separable, then every closed linear subspace \(Z\subseteq \B(E,F)\) containing \(\mathcal{F}(E,F)\) has the UBCP.
\end{corollary}
We choose projections \((P_n)\) implementing the \(\pi_\lambda\)-property.
Since \(P_nT\in\mathcal{F}(E,F)\subseteq Z\) for every \(T\in Z\), we have \(P_nZ\subseteq Z\),
and the conclusion follows from Corollary~\ref{cor:dual}.
In particular, if \(X\) is separable, uniformly convex and has the \(\pi_\lambda\)-property with \(1\leq\lambda<\Lambda_X\),
then every closed subspace of \(\B(X)\) containing \(\F(X)\) has the UBCP.
Indeed, \(X^*\) is separable by reflexivity, so Corollary~\ref{cor:finite} applies with \(E=F=X\).

We next use weak-star convergence and norm-preserving lifts to construct coverings of weak-star closed quotients.
\begin{theorem}\label{thm:quotient}
    Let \(Z=X^*\) and let \((Q_n)\subseteq\B(Z)\) satisfy
    \[Q_nz\xrightarrow{w^*}z\quad(z\in Z),\qquad \norm{Q_n}\leq M,\quad\text{and}\quad \norm{Q_n-tI_Z}\leq\mu,\]
    where \(M\geq1\), \(0<t<1\) and \(0<\mu<1\).
    Suppose that each \(Q_nZ\) is norm separable.
    Let \(J\subseteq Z\) be a weak-star closed linear subspace and let \(q:Z\to Z/J\) be the quotient map.
    For every \(\delta\) with \(0<\delta<(1-\mu)/t\), let \(a=M/t\).
    Then there are \(G_j\in\bigcup_n S_{q(Q_nZ)}\) such that
    \[S_{Z/J}\subseteq\bigcup_{j\geq1}B(aG_j,a-\delta).\]
\end{theorem}
\begin{proof}
    We define \(J_\perp=\{x\in X:z(x)=0\text{ for every }z\in J\}\).
    Restriction identifies \(Z/J\) isometrically with \((J_\perp)^*\).
    For \(w\in S_{Z/J}\), we choose \(z\in S_Z\) with \(qz=w\) by the Hahn--Banach theorem.
    Since \(qQ_nz\to w\) weak-star, we have by lower semicontinuity
    \[\liminf_n\norm{qQ_nz}\geq1.\]
    We choose \(n\) such that \(s:=\norm{qQ_nz}>\mu+t\delta\), and let \(A=qQ_nz/s\).
    Since \(s\leq M\), we have
    \[
        \begin{aligned}
            \norm{w-aA} &\leq\left\|q\left(z-\frac{1}{t}Q_nz\right)\right\| +a-\frac{s}{t}\\
            &\leq\frac{\mu}{t}+\frac{M-s}{t}<a-\delta.
        \end{aligned}
    \]
    Since each \(q(Q_nZ)\) is separable,
    we choose a countable dense subset of each nonempty \(S_{q(Q_nZ)}\) and enumerate their union as \((G_j)\).
    The strict inequality above remains valid after replacing \(A\) by a sufficiently close \(G_j\).
    This proves the covering.
    As in Lemma~\ref{lem:cover}, \(0<a-\delta<a\).
\end{proof}

\begin{corollary}\label{cor:operator-quotient}
    Under the assumptions of Corollary~\ref{cor:finite},
    every quotient \(\B(E,F)/J\) by a weak-star closed linear subspace \(J\) has the UBCP with the parameters of Theorem~\ref{thm:operator}.
    Here the weak-star topology is induced by \(E\widehat{\otimes}_\pi F^*\).
    The centres can be represented by finite-rank operators.
    Equivalently, the dual of every closed subspace of \(E\widehat{\otimes}_\pi F^*\) has these coverings.
\end{corollary}
\begin{proof}
    Since \(F\) is reflexive, we identify \((E\widehat{\otimes}_\pi F^*)^*\) isometrically with \(\B(E,F)\).
    We choose the projections \((P_n)\) and the parameters \(t,\mu\) from Theorem~\ref{thm:operator}, and define \(Q_nT=P_nT\).
    The images \(Q_n\B(E,F)\) are separable by Corollary~\ref{cor:dual}.
    For \(x\in E\) and \(f^*\in F^*\), we have
    \[\langle Q_nT,x\otimes f^*\rangle=f^*(P_nTx) \longrightarrow f^*(Tx)=\langle T,x\otimes f^*\rangle.\]
    By uniform boundedness, we extend the convergence to the projective tensor product.
    We can therefore apply Theorem~\ref{thm:quotient} with \(M=\lambda\).
    Its directions belong to quotient images of finite-rank operators.
    The final assertion follows by taking \(J=Y^\perp\) for a closed subspace \(Y\subseteq E\widehat{\otimes}_\pi F^*\).
\end{proof}

\begin{remark}\label{rem:lms-comparison}
    The weak-star truncation assumptions of Theorem~\ref{thm:quotient} and the uniform-exposure assumptions of \cite[Theorems~3.6--3.7]{LMS26} are not comparable for a fixed dual pair.
    Indeed, on \(\ell_1=c_0^*\) the coordinate projections satisfy
    \[\norm{Q_n}=1,\qquad \norm{Q_n-\tfrac{1}{2}I}=\tfrac{1}{2},\quad\text{and}\quad Q_nz\longrightarrow z.\]
    Their images are finite dimensional, whereas \(B_{c_0}\) has no strongly exposed points;
    see \cite[Remark~3.8]{LMS26}.

    For the other direction, work over the real field.
    Schachermayer's embedding theorem \cite[Theorem~4.6]{Schach83} gives a separable space \(W\) with property \(\alpha\) containing an isometric copy of \(C[0,1]\).
    Property \(\alpha\) supplies a uniformly strongly exposed family generating \(B_W\), so \cite[Theorem~3.6]{LMS26} gives the UBCP of \(W^*\).
    Restriction onto \(C[0,1]\) identifies a weak-star closed quotient of \(W^*\) with \(C[0,1]^*\),
    which fails the BCP \cite[Corollary~20]{GLM19}.
    Consequently \(W^*\), with its predual \(W\), cannot satisfy the hypotheses of Theorem~\ref{thm:quotient}.

    For a partition \(\N=A\sqcup B\) into infinite sets,
    the diagonal operators with symbols in \(c_0(A)\oplus_\infty\ell_\infty(B)\) form a closed subspace of \(\B(\ell_2)\) invariant under standard truncations,
    so Theorem~\ref{thm:operator} applies.
    Its unit ball has no extreme points, so it is not isometric to a dual space.
\end{remark}

\section{Applications to classical and local \texorpdfstring{\(L_p\)}{Lp}-spaces}
\label{sec:classical}
This section supplies projection approximations to which Theorem~\ref{thm:operator} applies.
Conditional expectations give the classical and vector-valued examples,
whereas almost-isometric finite-dimensional complementation yields the local \(L_p\) case.

\subsection{Classical and vector-valued \texorpdfstring{\(L_p\)}{Lp}-spaces}

\begin{lemma}\label{lem:bochner-pi}
    Let \(1\leq p<\infty\), let \(X\) be a Banach space with the \(\pi_\lambda\)-property, and suppose that \(Y=L_p(\mu;X)\) is separable.
    Then \(Y\) has the \(\pi_\lambda\)-property.
\end{lemma}

\begin{corollary}\label{cor:bochner}
    Let \(1<p<\infty\), let \(X\) be uniformly convex with the \(\pi_\lambda\)-property,
    and suppose that \(Y=L_p(\mu;X)\) is separable and \(1\leq\lambda<\Lambda_Y\).
    Let \((Q_n)\) be finite-rank projections on \(Y\) such that \(\sup_n\norm{Q_n}\leq\lambda\) and \(Q_ny\to y\) for every \(y\in Y\).
    If \(E^*\) is separable and \(Z\subseteq\B(E,Y)\) is a closed subspace satisfying
    \[Q_nT\in Z\qquad(T\in Z,\ n\geq1),\]
    then \(Z\) has the UBCP.
    In particular, the conclusion holds whenever \(\F(E,Y)\subseteq Z\).
\end{corollary}

Since \(Y\) is uniformly convex \cite{LT}, we obtain the conclusion by applying Corollary~\ref{cor:dual} to \((Q_n)\).

The scalar case is obtained by taking \(X=\mathbb{K}\), and includes \(L_p[0,1]\).
When \(X\) has the metric \(\pi\)-property the threshold condition is automatic.
Corollary~\ref{cor:operator-quotient} also applies to every quotient \(\B(E,Y)/J\) by a weak-star closed linear subspace \(J\).

\subsection{Local \texorpdfstring{\(L_p\)}{Lp} structure}
We next obtain near-metric projection approximations from local finite-dimensional structure.
Recall that the Banach--Mazur distance between two isomorphic Banach spaces is
\[d_{\mathrm{BM}}(X,Y) =\inf\{\norm{T}\norm{T^{-1}}:T:X\to Y\text{ is an isomorphism}\}.\]

\begin{definition}
    Let \(1<p<\infty\) and \(C\geq1\).
    A Banach space \(F\) is an \(\mathcal{L}_{p,C+}\)-space if,
    for every finite-dimensional subspace \(H\subseteq F\) and every \(\varepsilon>0\),
    there is a finite-dimensional subspace \(G\subseteq F\) such that
    \[H\subseteq G\quad\text{and}\quad d_{\mathrm{BM}}(G,\ell_p^{\dim G})<C+\varepsilon.\]
\end{definition}

We use the following finite-dimensional consequence of Schechtman's almost-isometric complementation theorem \cite[Theorem~1]{Schechtman}.

\begin{lemma}\label{lem:local-stability}
    Let \(1<p<\infty\).
    There exist \(r_p>1\) and a nondecreasing function \(\Gamma_p:[1,r_p)\to[1,\infty)\) with \(\lim_{r\downarrow1}\Gamma_p(r)=1\) such that,
    whenever \(Y\subseteq\ell_p^N\) and
    \[d_{\mathrm{BM}}(Y,\ell_p^{\dim Y})<r<r_p,\]
    there is a projection \(R:\ell_p^N\to Y\) satisfying \(\norm{R}\leq\Gamma_p(r)\).
\end{lemma}

We embed \(\ell_p^N\) isometrically into \(L_p[0,1]\) as a space of step functions.
By Schechtman's theorem \cite[Theorem~1]{Schechtman}, the subspace \(Y\) is complemented in \(L_p[0,1]\),
with a projection bound tending to one as the distortion tends to one.
We restrict this projection to \(\ell_p^N\) and take a nondecreasing majorant of the bound, decreasing \(r_p\) if necessary.
This gives \(\Gamma_p\) as stated.
For \(p=2\), we may take \(\Gamma_2(r)=1\) by using orthogonal projections.

\begin{proposition}\label{prop:local-pi}
    Let \(1<p<\infty\) and let \(F\) be a separable \(\mathcal{L}_{p,C+}\)-space.
    Suppose that \(\eta>0\) and \(\rho\) satisfy
    \[(C+\eta)^2<\rho<r_p.\]
    Then \(F\) has the \(\pi_M\)-property for \(M=(C+\eta)\Gamma_p(\rho)\).
\end{proposition}

\begin{proof}
    Let \((x_n)\) be dense in \(F\).
    We inductively choose finite-dimensional subspaces \(G_1\subseteq G_2\subseteq\cdots\subseteq F\) such that \(x_1,\ldots,x_n\in G_n\) and
    \[d_{\mathrm{BM}}(G_n,\ell_p^{\dim G_n})<C+\eta.\]
    Their union is dense in \(F\).
    We fix \(n<k\) and choose an isomorphism \(T_k:G_k\to\ell_p^{\dim G_k}\) with \(\norm{T_k}\norm{T_k^{-1}}<C+\eta\).
    Then
    \[
        \begin{aligned}
            d_{\mathrm{BM}}(T_kG_n,\ell_p^{\dim G_n}) &\leq d_{\mathrm{BM}}(T_kG_n,G_n) d_{\mathrm{BM}}(G_n,\ell_p^{\dim G_n})\\
            &<(C+\eta)^2<\rho.
        \end{aligned}
    \]
    By Lemma~\ref{lem:local-stability},
    we choose a projection \(S_{n,k}:\ell_p^{\dim G_k}\to T_kG_n\) with \(\norm{S_{n,k}}\leq\Gamma_p(\rho)\).
    Consequently,
    \[R_{n,k}=T_k^{-1}S_{n,k}T_k:G_k\to G_n\]
    is a projection and \(\norm{R_{n,k}}\leq M\).

    For fixed \(n\) and \(m\geq n\), the restrictions \(R_{n,k}|_{G_m}\), \(k\geq m\),
    belong to the finite-dimensional space \(\B(G_m,G_n)\) and have norm at most \(M\).
    By a diagonal compactness argument,
    we choose a subsequence \((k_j)\) tending to infinity for which \(R_{n,k_j}|_{G_m}\) converges for every fixed \(m\geq n\).
    We define
    \[P_nx=\lim_{j\to\infty}R_{n,k_j}x\quad\text{for } x\in\bigcup_mG_m.\]
    These limits are consistent on the increasing subspaces.
    Hence \(P_n\) is linear, \(\norm{P_n}\leq M\), and \(P_n|_{G_n}=I_{G_n}\).
    It extends to a projection from \(F\) onto \(G_n\) with the same norm bound.

    If \(x\in G_m\), then \(P_nx=x\) for every \(n\geq m\).
    Since \(\bigcup_mG_m\) is dense and \(\sup_n\norm{P_n}\leq M\), we obtain \(P_nx\to x\) for every \(x\in F\).
\end{proof}

\begin{corollary}\label{cor:local-Lp}
    Let \(1<p<\infty\) and let \(F\) be a separable uniformly convex \(\mathcal{L}_{p,C+}\)-space.
    Suppose that
    \[(C+\eta)^2<\rho<r_p\quad\text{and}\quad (C+\eta)\Gamma_p(\rho)<\Lambda_F\]
    for some \(\eta>0\) and \(\rho\).
    Let \((P_n)\) be the projections supplied by Proposition~\ref{prop:local-pi}.
    If \(E^*\) is separable and \(Z\subseteq\B(E,F)\) is closed with \(P_nZ\subseteq Z\) for every \(n\), then \(Z\) has the UBCP.
    In particular, the conclusion holds for every such \(Z\) containing \(\F(E,F)\).
\end{corollary}

We apply Corollary~\ref{cor:dual} to the projections supplied by Proposition~\ref{prop:local-pi},
whose norms are bounded by \((C+\eta)\Gamma_p(\rho)<\Lambda_F\).

In particular, the conclusion applies to every separable \(\mathcal{L}_{p,1+}\)-space.
Indeed, apply the defining local approximation to the span of any two vectors.
Letting the distortion decrease to one transfers Clarkson's inequalities \cite{Clarkson} from the finite-dimensional \(\ell_p\)-spaces to \(F\),
so \(F\) is uniformly convex.
Since \(\Gamma_p(r)\to1\) as \(r\downarrow1\), choose \(\rho>1\) sufficiently close to one and then \(\eta>0\) sufficiently small so that
\[(1+\eta)^2<\rho<r_p\quad\text{and}\quad (1+\eta)\Gamma_p(\rho)<\Lambda_F.\]
Corollary~\ref{cor:local-Lp} now applies.

\section{Noncommutative \texorpdfstring{\(L_p\)}{Lp}-spaces and completely bounded maps}

For noncommutative \(L_p\)-spaces, we obtain uniform translation estimates from a common lower bound for the moduli of convexity of their Schatten amplifications.

We use the standard operator-space conventions of Effros--Ruan and Pisier \cite{ER00,Pisier2003}.
Noncommutative \(L_p\)-spaces carry their canonical operator-space structures.
If \(E,F\) are operator spaces, then \(\CB(E,F)\) has its completely bounded norm and the canonical matrix norms
\[M_d(\CB(E,F))=\CB(E,M_d(F)).\]

We first record the postcomposition form of the sufficient condition for uniform ball coverings.

\begin{lemma}\label{lem:cb-cover}
    Let \(E,F\) be operator spaces, and let \(Z\subseteq\CB(E,F)\) be a closed linear subspace.
    Suppose that \((A_n)\) is a sequence of finite-rank completely bounded maps on \(F\) such that
    \[A_ny\longrightarrow y\quad(y\in F),\qquad \norm{A_n}_{\mathrm{cb}}\leq M,\quad\text{and}\quad \norm{A_n-t\Id_F}_{\mathrm{cb}}\leq\mu,\]
    where \(M\geq1\), \(0<t<1\) and \(0<\mu<1\).
    Assume that \(A_nZ\subseteq Z\) and that \(A_nZ\) is separable in the cb norm for every \(n\).
    For every \(\delta\) with \(0<\delta<(1-\mu)/t\), let \(a=M/t\).
    Then, for every \(d\geq1\), there is a sequence \((G_{d,j})_{j\geq1}\subseteq S_{M_d(Z)}\) such that
    \[S_{M_d(Z)}\subseteq\bigcup_{j\geq1}B(aG_{d,j},a-\delta).\]
    The entries of every \(G_{d,j}\) have finite rank.
\end{lemma}

\begin{proof}
    We define \(Q_n:Z\to Z\) by \(Q_n(T)=A_nT\).
    Postcomposition with a completely bounded map is completely bounded with no increase in its cb norm.
    Hence we have
    \[\norm{Q_n}_{\mathrm{cb}}\leq M\quad\text{and}\quad \norm{Q_n-t\Id_Z}_{\mathrm{cb}}\leq\mu.\]
    We fix \(d\geq1\), and let \(Q_n^{(d)}\) act entrywise on \(M_d(Z)\).
    We verify the norm detection condition in Lemma~\ref{lem:cover}.
    Let \(U=[T_{ij}]\in S_{M_d(Z)}\).
    We regard \(U\) as a map from \(E\) to \(M_d(F)\), given by \(U(x)=[T_{ij}x]\).
    For \(\varepsilon>0\), we choose \(k\geq1\) and \(x\in B_{M_k(E)}\) such that
    \[\norm{U^{(k)}(x)}>1-\varepsilon.\]
    For each fixed \(m\) and each \([y_{ij}]\in M_m(F)\), the convergence \(A_ny_{ij}\to y_{ij}\) implies
    \[\norm{[A_ny_{ij}-y_{ij}]}_{M_m(F)} \leq\sum_{i,j=1}^m\norm{A_ny_{ij}-y_{ij}}\longrightarrow0.\]
    We apply this with \(m=kd\) to \(U^{(k)}(x)\).
    Identifying \(M_k(M_d(F))\) with \(M_{kd}(F)\), we obtain
    \[
        \begin{split}
            \liminf_n\norm{Q_n^{(d)}U} \geq\lim_n\norm{A_n^{(kd)}U^{(k)}(x)}=\norm{U^{(k)}(x)}>1-\varepsilon.
        \end{split}
    \]
    Letting \(\varepsilon\downarrow0\), we obtain \(\liminf_n\norm{Q_n^{(d)}U}\geq1\).
    Moreover, \(Q_n^{(d)}M_d(Z)=M_d(A_nZ)\) is separable.
    Since
    \[\norm{Q_n^{(d)}}\leq M\quad\text{and}\quad \norm{Q_n^{(d)}-t\Id_{M_d(Z)}}\leq\mu,\]
    we obtain the covering from Lemma~\ref{lem:cover}, with directions in \(\bigcup_nS_{M_d(A_nZ)}\).
    Every entry of such a direction has range in a finite-dimensional space \(A_nF\), and hence has finite rank.
\end{proof}

\subsection{Projection estimates and matrix coverings}
Let \(1<p<\infty\), and let \(\M\) be a von Neumann algebra equipped with a normal faithful semifinite trace \(\tau\).
We use \(S_p^k[F]\) for the vector-valued Schatten space.
For \(F=L_p(\M)\) and \(A\in\CB(F,F)\), we have
\begin{equation}\label{eq:schatten-amplification}
    \begin{split}
        S_p^k[L_p(\M)] &=L_p(M_k\overline{\otimes}\M),\qquad \norm{A}_{\mathrm{cb}} =\sup_{k\geq1}\norm{\Id_{S_p^k}\otimes A: S_p^k[F]\longrightarrow S_p^k[F]};
    \end{split}
\end{equation}
see \cite[Chapter~1, Lemma~1.7, and Chapter~3, formula~\((3.6)'\)]{Pisier98}.
The trace in the first identity is \(\operatorname{Tr}_k\otimes\tau\).
The noncommutative Clarkson inequalities give a common lower bound for the moduli of convexity of these spaces.
Let \(q_p=\max\{p,p/(p-1)\}\) and define
\[\delta_p(\varepsilon)=1-\left(1-(\varepsilon/2)^{q_p}\right)^{1/q_p}.\]
For \(x,y\in B_{L_p(\M)}\), the Clarkson inequalities give
\[\left\|\frac{x+y}{2}\right\|_p^{q_p} +\left\|\frac{x-y}{2}\right\|_p^{q_p}\leq1;\]
see \cite[Theorem~5.1]{PisierXu03}.
It follows that
\[\delta_{L_p(\M)}(\varepsilon)\geq\delta_p(\varepsilon) \qquad(0<\varepsilon\leq2);\]
this bound also holds on every \(L_p(M_k\overline{\otimes}\M)\).
Let \(t_p=\delta_p(1/2)/8\).

\begin{lemma}\label{lem:cb-projection}
    Let \(P\) be a completely bounded projection on \(F=L_p(\M)\) such that \(\norm{P}_{\mathrm{cb}}\leq\lambda<1+t_p/4\),
    where \(\lambda\geq1\).
    Then
    \[\norm{P-t_p\Id_F}_{\mathrm{cb}} \leq\lambda-\frac{t_p}{2}<1.\]
\end{lemma}

\begin{proof}
    For every \(k\geq1\), the map \(\Id_{S_p^k}\otimes P\) is a projection on \(L_p(M_k\overline{\otimes}\M)\) of norm at most \(\lambda\).
    We apply Lemma~\ref{lem:projection} with \(\eta=1/2\), \(t=t_p\) and the common modulus lower bound \(\delta_p(1/2)\).
    The admissibility conditions hold because \(t_p=\delta_p(1/2)/8\) and
    \[\lambda-1<\frac{t_p}{4} =\frac{t_p(1-1/2)}{2}.\]
    Thus we have
    \[\norm{\Id_{S_p^k}\otimes(P-t_p\Id_F)} \leq\lambda-\frac{t_p}{2}.\]
    Taking the supremum over \(k\) and using \eqref{eq:schatten-amplification}, we obtain the assertion.
\end{proof}

\begin{theorem}\label{thm:cb}
    Let \(F=L_p(\M)\), and suppose that \((P_n)\) is a sequence of finite-rank projections on \(F\) such that \(P_ny\to y\) for every \(y\in F\) and \(\sup_n\norm{P_n}_{\mathrm{cb}}\leq\lambda\) for some \(1\leq\lambda<1+t_p/4\).
    Let \(E\) be an operator space and let \(Z\) be a closed linear subspace of \(\CB(E,F)\) such that \(P_nZ\subseteq Z\) and \(P_nZ\) is separable in the cb norm for every \(n\).
    For every \(\delta\) with \(0<\delta<\frac{1}{2}-\frac{\lambda-1}{t_p}\), let \(a=\lambda/t_p\).
    Then, for every \(d\geq1\), there is a sequence \((G_{d,j})_{j\geq1}\subseteq S_{M_d(Z)}\) such that
    \[S_{M_d(Z)}\subseteq\bigcup_{j\geq1}B(aG_{d,j},a-\delta).\]
    Every \(G_{d,j}\) has finite-rank entries.
    The covering parameters depend only on \(p\), \(\lambda\) and \(\delta\).
\end{theorem}

By Lemma~\ref{lem:cb-projection}, we have \(\norm{P_n-t_p\Id_F}_{\mathrm{cb}}\leq\lambda-t_p/2\) for every \(n\).
We obtain the conclusion by applying Lemma~\ref{lem:cb-cover} with \(A_n=P_n\), \(M=\lambda\), \(t=t_p\) and \(\mu=\lambda-t_p/2\).

\begin{corollary}\label{cor:cb-metric}
    Under the assumptions of Theorem~\ref{thm:cb}, suppose that \(\norm{P_n}_{\mathrm{cb}}\leq1\) for every \(n\).
    For every \(0<\delta<1\) there is \(a=a(p,\delta)>1\) such that, for every \(d\geq1\),
    \[S_{M_d(Z)}\subseteq\bigcup_{j\geq1}B(aG_{d,j},a-\delta),\]
    where \(\norm{G_{d,j}}=1\) for every \(j\geq1\).
    Each \(G_{d,j}\) has finite-rank entries.
\end{corollary}

We choose \(0<\eta<1-\delta\) and \(0<t<\delta_p(\eta)/4\).
By Lemma~\ref{lem:projection} at each Schatten level with \(\lambda=1\), we have
\[\norm{P_n-t\Id_F}_{\mathrm{cb}}\leq1-t(1-\eta).\]
Since \(\delta<1-\eta\), we can apply Lemma~\ref{lem:cb-cover} with \(M=1\) and \(a=1/t\).

\begin{corollary}\label{cor:cb-models}
    Let \(1<p<\infty\), and let \(F=S_p\) or \(F=L_p(\mathcal{R})\), where \(\mathcal{R}\) is the hyperfinite \(\mathrm{II}_1\) factor.
    Let \(E\) be an operator space with separable dual.
    If \(Z\subseteq\CB(E,F)\) is a closed linear subspace containing all finite-rank maps from \(E\) to \(F\),
    then the conclusion of Corollary~\ref{cor:cb-metric} holds.
\end{corollary}

\begin{proof}
    For \(F=S_p\), we let \(e_n\) be the projection onto the first \(n\) coordinates of \(\ell_2\), and define \(P_nx=e_nxe_n\).
    These maps are finite-rank completely contractive projections and \(P_nx\to x\) in \(S_p\).
    For \(F=L_p(\mathcal{R})\), we choose an increasing sequence of finite-dimensional unital matrix subalgebras \(\mathcal{R}_n\) with weakly dense union.
    The trace-preserving conditional expectations \(P_n:L_p(\mathcal{R})\to L_p(\mathcal{R}_n)\) are finite-rank completely contractive projections.
    Since \(\bigcup_n\mathcal{R}_n\) is dense in \(L_p(\mathcal{R})\), we have \(P_nx\to x\).

    In either case \(P_nT\) is finite rank, so \(P_nZ\subseteq Z\).
    To check separability, we fix a basis \(y_1,\ldots,y_s\) of \(P_nF\).
    Every map \(T:E\to P_nF\) has a unique representation \(T=\sum_{j=1}^sf_j\otimes y_j\), with \(f_j\in E^*\), and
    \[\norm{T}_{\mathrm{cb}} \leq\sum_{j=1}^s\norm{f_j}\norm{y_j}.\]
    The coordinate functionals give a continuous inverse map.
    Since \(E^*\) is separable, so is \(\CB(E,P_nF)\), and we can apply Corollary~\ref{cor:cb-metric}.
\end{proof}

\subsection{Weak-star quotients and trace-class subspaces}
The operator space quotient of a dual operator space by a weak-star closed subspace is again a dual operator space.
Its matrix norms satisfy
\[M_d(Z/J)=M_d(Z)/M_d(J).\]
The restriction map from \(X^*\) onto \(Y^*\), for an operator subspace \(Y\subseteq X\),
gives the corresponding completely isometric identification \(X^*/Y^\perp=Y^*\);
see \cite{ER00}.

\begin{proposition}\label{prop:cb-quotient}
    Let \(Z\) be a dual operator space and let \(J\subseteq Z\) be a weak-star closed linear subspace.
    Suppose that the hypotheses of Theorem~\ref{thm:quotient} hold, with
    \[\norm{Q_n}_{\mathrm{cb}}\leq M\quad\text{and}\quad \norm{Q_n-tI_Z}_{\mathrm{cb}}\leq\mu<1.\]
    Then for every \(0<\delta<(1-\mu)/t\) and every \(d\geq1\) there are \(G_{d,j}\in\bigcup_n S_{M_d(q(Q_nZ))}\) such that, with \(a=M/t\),
    \[S_{M_d(Z/J)}\subseteq\bigcup_{j\geq1}B(aG_{d,j},a-\delta).\]
\end{proposition}
\begin{proof}
    We fix \(d\geq1\).
    Then \(Q_n^{(d)}U\to U\) weak-star and \(Q_n^{(d)}M_d(Z)=M_d(Q_nZ)\) is separable.
    By the cb bounds in the hypotheses, we have the same operator bounds for \(Q_n^{(d)}\).
    Since \(M_d(J)\) is weak-star closed,
    we obtain a norm-one lift of every unit vector in the quotient by taking a weak-star cluster point of lifts whose norms tend to one.
    We now apply Theorem~\ref{thm:quotient} at this matrix level.
    All its constants are independent of \(d\).
\end{proof}

\begin{corollary}\label{cor:trace-duals}
    Let \(Y\) be a closed subspace of \(S_1(\ell_2)\), with its inherited operator space structure.
    For every \(0<\varepsilon<1\) and every \(d\geq1\), there is a sequence \((G_{d,j})\subseteq S_{M_d(Y^*)}\) such that
    \[S_{M_d(Y^*)}\subseteq\bigcup_{j\geq1}B(2G_{d,j},1+\varepsilon).\]
    Each \(G_{d,j}\) has a representative consisting of finite-rank operators in \(M_d(\B(\ell_2))\).
    The centre norm and radius are independent of \(Y\) and \(d\).
\end{corollary}
\begin{proof}
    We identify \(Y^*\) completely isometrically with \(\B(\ell_2)/Y^\perp\).
    Let \(e_n\) be the orthogonal projection onto the first \(n\) coordinates and define \(Q_nT=e_nT\).
    By left multiplication, we have
    \[\norm{Q_n}_{\mathrm{cb}}=1\quad\text{and}\quad \left\|Q_n-\tfrac{1}{2}I\right\|_{\mathrm{cb}} =\left\|e_n-\tfrac{1}{2}I_{\ell_2}\right\|=\tfrac{1}{2}.\]
    Since \(xe_n\to x\) in trace norm for every \(x\in S_1\), we have \(Q_nT\to T\) ultraweakly by the trace pairing.
    The range \(Q_n\B(\ell_2)=\B(\ell_2,e_n\ell_2)\) is norm separable, and each of its elements has finite rank.
    We apply Proposition~\ref{prop:cb-quotient} with \(M=1\), \(t=\mu=1/2\) and \(\delta=1-\varepsilon\).
    Its directions lie in the matrix images of these ranges under the quotient map, which proves the assertion about representatives.
\end{proof}

\subsection{Operator-valued Schur multipliers}
A corner identity gives the translation estimate on vector-valued Schatten spaces.

Let \(E,F\) be operator spaces and \(1\leq p,q\leq\infty\), where \(S_\infty[E]=\mathcal{K}(\ell_2)\otimes_{\min}E\).
For a matrix \(A=(A_{ij})\) of maps from \(E\) to \(F\), define
\[M_A((x_{ij}))=(A_{ij}x_{ij})\]
on finite matrices.
Given a cb-norm separable closed subspace \(\mathcal{V}\subseteq\CB(E,F)\),
denote by \(\mathfrak{M}_{p,q}^{\mathrm{cb}}(\mathcal{V})\) the space of these multipliers which extend completely boundedly from \(S_p[E]\) to \(S_q[F]\),
with \(A_{ij}\in\mathcal{V}\) and the inherited cb norm.
Define \(A^{(n)}_{ij}=A_{ij}\mathbf{1}_{\{i,j\leq n\}}\).

Let \(e_n\) be the projection onto the first \(n\) coordinates and define \(C_ny=e_nye_n\) on \(S_q[F]\).
Scalar left and right multiplication are completely contractive for scalar contractions,
and multiplication by unitaries is completely isometric;
see \cite[Chapter~1]{Pisier98}.
For \(\omega=\exp(2\pi i/3)\) and \(U_n=e_n+\omega(1-e_n)\), block multiplication gives
\[3C_ny-y=U_nyU_n+U_n^*yU_n^*.\]
Consequently,
\[\norm{C_n}_{\mathrm{cb}}\leq1\quad\text{and}\quad \norm{C_n-\tfrac{1}{3}I}_{\mathrm{cb}}\leq\tfrac{2}{3}.\]

\begin{theorem}\label{thm:operator-schur}
    Let \(Z\) be a closed linear subspace of \(\mathfrak{M}_{p,q}^{\mathrm{cb}}(\mathcal{V})\) such that \(M_{A^{(n)}}\in Z\) whenever \(M_A\in Z\) and \(n\geq1\).
    For every \(0<\varepsilon<1\) and every \(d\geq1\), there are \(G_{d,j}\in S_{M_d(Z)}\) with finite matrix support such that
    \[S_{M_d(Z)}\subseteq\bigcup_{j\geq1}B(3G_{d,j},2+\varepsilon).\]
    The parameters are independent of \(p,q,E,F,\mathcal{V}\) and \(d\).
\end{theorem}
\begin{proof}
    We define \(Q_nM_A=C_nM_A=M_{A^{(n)}}\).
    By the estimates for \(C_n\), we have
    \[\norm{Q_n}_{\mathrm{cb}}\leq1\quad\text{and}\quad \norm{Q_n-\tfrac{1}{3}I_Z}_{\mathrm{cb}}\leq\tfrac{2}{3}.\]
    Coordinate insertion and extraction are completely contractive.
    Thus, on finite corners, we have
    \[\max_{i,j\leq n}\norm{A_{ij}}_{\mathrm{cb}} \leq\norm{M_{A^{(n)}}}_{\mathrm{cb}} \leq\sum_{i,j\leq n}\norm{A_{ij}}_{\mathrm{cb}}.\]
    It follows that \(Q_nZ\), as a subspace of \(\mathcal{V}^{n^2}\) with an equivalent norm, is separable.

    Finite matrices are dense in \(S_q[F]\), including \(S_\infty[F]\), and hence \(C_ny\to y\) for every \(y\in S_q[F]\).
    We fix \(d\geq1\) and \(U\in M_d(Z)\).
    At a fixed amplification level, pointwise convergence of \(C_n\) on each matrix entry gives norm convergence on every matrix test for \(U\).
    By taking such tests arbitrarily close to the cb norm, we obtain
    \[\liminf_n\norm{Q_n^{(d)}U}\geq\norm{U}.\]
    We apply Lemma~\ref{lem:cover} to \(M_d(Z)\) with \(M=1\), \(t=1/3\), \(\mu=2/3\) and \(\delta=1-\varepsilon\).
    Its covering directions belong to \(\bigcup_nS_{M_d(Q_nZ)}\) and therefore have finite matrix support.
\end{proof}

The same argument gives the bounded version for a norm-separable closed \(\mathcal{V}\subseteq\B(E,F)\), with the operator norm.

Let \(\mathfrak{M}_p\) and \(\mathfrak{M}_p^{\mathrm{cb}}\) denote the scalar Schur multiplier spaces on \(S_p\),
with their operator and cb norms.
Taking \(E=F=\C\), \(\mathcal{V}=\C\) and \(p=q\) gives the following.
\begin{corollary}\label{cor:schur}
    Let \(1\leq p\leq\infty\) and let \(Z\) be a closed subspace of \(\mathfrak{M}_p\) invariant under finite corner truncations.
    For every \(0<\varepsilon<1\), there are finite-support \(G_j\in S_Z\) such that
    \[S_Z\subseteq\bigcup_{j\geq1}B(3G_j,2+\varepsilon).\]
    For a closed corner-invariant subspace of \(\mathfrak{M}_p^{\mathrm{cb}}\), the same centre norm and radius work at every matrix level.
\end{corollary}
For \(p=\infty\), a scalar multiplier on \(\mathcal{K}(\ell_2)\) has a unique normal Schur extension to \(\B(\ell_2)\) with the same operator and cb norms.
This follows from finite corner tests and \(\norm{x}=\sup_n\norm{e_nxe_n}\) on \(\B(\ell_2)\).
Any prescribed matrix support is invariant under corner truncation.

\begin{example}\label{ex:schur-no-ase}
    Let \(I\) be uncountable, equip \(c_0(I)\) with its canonical operator space structure, and let \(X=S_p[c_0(I)]\), \(1\leq p\leq\infty\).
    The scalar Schur multiplier spaces on \(X\) have the preceding uniform coverings,
    although \(X\) fails the BCP and none of their nonzero members is an absolutely strongly exposing operator on \(X\).
\end{example}
\begin{proof}
    We apply Theorem~\ref{thm:operator-schur} and its bounded version with \(\mathcal{V}=\C I_{c_0(I)}\).
    To show that \(X\) fails the BCP, we fix an arbitrary sequence \((z_n)\subseteq X\).
    By density of finite tensors, we choose a countable \(J\subseteq I\) supporting every \(z_n\) in the \(c_0(I)\) coordinate.
    We choose \(i_0\in I\setminus J\) and let \(x=e_{11}\otimes e_{i_0}\).
    The induced coordinate projection \(\Pi_J\) is contractive, fixes every \(z_n\) and annihilates \(x\).
    Hence we have
    \[\norm{x}=1\quad\text{and}\quad \norm{x-z_n}\geq\norm{\Pi_J(x-z_n)}=\norm{z_n},\]
    which excludes a ball-covering.

    An operator \(T\) is absolutely strongly exposing if every sequence \((x_n)\subseteq B_X\) with \(\norm{Tx_n}\to\norm{T}\) converges to a fixed unit vector after suitable phase changes;
    see \cite[Section~3]{MPSS26} for the real version of this definition.
    For a nonzero scalar multiplier \(M_A\),
    we choose unit vectors \(x_n\) with finite internal coordinate support such that \(\norm{M_Ax_n}\to\norm{M_A}\).
    Permutations of \(I\) induce complete isometries on \(X\) and commute with \(M_A\).
    We can therefore arrange that the internal supports of the \(x_n\) are pairwise disjoint.
    By coordinate projections, we then have
    \[\norm{x_n-\zeta x_m}\geq1\qquad(n\ne m,\ |\zeta|=1).\]
    No phase-adjusted subsequence can converge, so \(M_A\) is not absolutely strongly exposing.
\end{proof}
These multiplier spaces are nonseparable: for \(\epsilon\in\{-1,1\}^{\N}\) with \(\epsilon_1=1\),
the symbols \(a^\epsilon_{ij}=\epsilon_i\epsilon_j\) induce complete isometries,
and distinct such multipliers have distance \(2\) in both norms.

On the underlying real Banach space, neither alternative in \cite[Theorem~3.5]{MPSS26} applies to the pair \((X,X)\),
since both imply that the range has the BCP \cite[Remark~3.6]{MPSS26}.
The separability assumption on the domain in \cite[Theorem~4.5]{LMS26} also fails.
For every \(0<\varepsilon<1\), Theorem~\ref{thm:operator-schur} gives coverings of the completely bounded multiplier space at every matrix level,
with finite-support centres of norm \(3\) and radius \(2+\varepsilon\).

For the reverse implication, let \(Y\) be the real space with the BCP but without the strong ball-covering property in \cite[Proposition~3.12]{LMS26}.
Then \cite[Theorem~3.5(i)]{MPSS26} applies to \(\B(\mathbb{R},Y)=Y\), but the hypotheses of Lemma~\ref{lem:cover} cannot hold on \(Y\),
since they would imply the UBCP.
Thus the hypotheses of Lemma~\ref{lem:cover} and \cite[Theorem~3.5]{MPSS26} do not imply one another.

\subsection{Fourier multipliers and factorizable maps}
For Schur--Fourier transference in the discrete amenable setting,
see Neuwirth--Ricard \cite{NeuwirthRicard2011} and Caspers--de la Salle \cite{CaspersSalle}.
We use Fourier truncations that preserve the symbol structure.
Their translation estimates follow from factorization through conditional expectations.

Let \(\Gamma\) be a countable discrete group and let \(\mathcal{M}=\operatorname{VN}(\Gamma)\).
Denote its left regular unitaries by \(u_g\) and its canonical trace by \(\tau\).
For \(1\leq p<\infty\), let \(\mathsf{FM}_p(\Gamma)\) be the space of bounded Fourier multipliers
\[T_m(u_g)=m(g)u_g\qquad(g\in\Gamma)\]
on \(L_p(\mathcal{M})\), with its norm in \(\B(L_p(\mathcal{M}))\).
Let \(\mathsf{FM}^{\mathrm{cb}}_p(\Gamma)\) denote the space of completely bounded Fourier multipliers,
with the operator space structure inherited from \(\operatorname{CB}(L_p(\mathcal{M}),L_p(\mathcal{M}))\).
These are closed subspaces, since their defining relations on each \(u_g\) are preserved under norm limits.
At \(p=\infty\), let \(\mathsf{FM}_\infty(\Gamma)\) consist of the normal bounded Fourier multipliers on \(\mathcal{M}\),
and let \(\mathsf{FM}^{\mathrm{cb}}_\infty(\Gamma)\) be its completely bounded subspace with the cb norm.
See \cite{CaspersSalle} for background.

For \(1<p<\infty\), the ordinary BCP of prescribed-support Fourier multiplier spaces on every countable group also follows from the complex analogue of the coordinate method in \cite[Lemmas~3.2 and~3.4]{MPSS26}.
Indeed, the operator \(R_g\) defined by \(R_gx=\tau(xu_g^*)u_g\) is absolutely strongly exposing on \(L_p(\mathcal{M})\).
Moreover, \(R_g\perp_B T_m\) if and only if \(m(g)=0\).
Here \(S\perp_B T\) denotes complex Birkhoff--James orthogonality, that is,
\(\norm{S+\alpha T}\geq\norm{S}\) for every \(\alpha\in\mathbb{C}\).
A countable family of phase multiples gives the corresponding sufficient condition for ball coverings.
The coordinates separate multipliers, and \(R_g\) belongs to every prescribed-support space with \(g\) in its support.

The trace adjoint \(J^\dagger\) of a trace-preserving embedding \(J:\mathcal{M}\to\mathcal{N}\) is its conditional-expectation retraction:
\(JJ^\dagger\) is the expectation onto \(J(\mathcal{M})\) and \(J^\dagger J=I\).
On every \(L_p\), the maps \(J\) and \(J^\dagger\) are, respectively, a complete isometry and a complete contraction.

We use Ricard's factorization \cite{Ricard08} of unital completely positive Fourier multipliers with real-valued symbols.
For the general factorization problem for Markov maps, see Haagerup--Musat \cite{HaagerupMusat2011}.
Squaring the multiplier gives a compression \(J^\dagger PJ\) of a conditional expectation.
Consequently,
\[J^\dagger PJ-tI=J^\dagger(P-tI)J,\]
so Lemma~\ref{lem:projection} gives a translation estimate after compression.

\begin{lemma}\label{lem:fourier-factorization}
    Let \(1<p<\infty\), \(0<\eta<1\), and \(0<t<\delta_p(\eta)/4\).
    If \(\phi:\Gamma\to\mathbb{R}\) is positive definite and \(\phi(e)=1\), then \(R=T_{\phi^2}\) satisfies
    \[\norm{R}_{\mathrm{cb}}\leq1\quad\text{and}\quad \norm{R-tI}_{\mathrm{cb}}\leq1-t(1-\eta).\]
\end{lemma}

\begin{proof}
    By Ricard's factorization theorem \cite[p.~4370, Corollary]{Ricard08},
    we choose a finite tracial von Neumann algebra \(\mathcal{N}\) and trace-preserving normal embeddings \(J_0,J_1:\mathcal{M}\to\mathcal{N}\) such that \(T_\phi=J_1^\dagger J_0.\)
    The trace adjoint of this map is \(J_0^\dagger J_1\).
    Since \(\phi\) is real and positive definite, we have \(\phi(g^{-1})=\phi(g)\), so \(T_\phi\) equals its trace adjoint.
    We define \(P=J_1J_1^\dagger\), the trace-preserving conditional expectation onto \(J_1(\mathcal{M})\).
    Then
    \begin{equation}\label{eq:fourier-factorization}
        R=T_{\phi^2}=T_\phi\circ T_\phi=J_0^\dagger PJ_0.
    \end{equation}
    Since \(J_0^\dagger J_0=I\), we also have
    \[R-tI=J_0^\dagger(P-tI)J_0.\]
    We apply Lemma~\ref{lem:projection} to each Schatten amplification of \(P\), using \(\lambda=1\) and the common modulus bound \(\delta_p\).
    The constant in that lemma is \(1-t(1-\eta)\).
    Hence we have\(\norm{P-tI}_{\mathrm{cb}}\leq1-t(1-\eta).\)
    Since \(J_0^\dagger\) and \(J_0\) are completely contractive, we obtain the required estimate for \(R-tI\).
    By \eqref{eq:fourier-factorization}, we also have \(\norm{R}_{\mathrm{cb}}\leq1\).
\end{proof}

Suppose now that \(\Gamma\) is amenable.
Fix a left F\o lner sequence \((F_n)\) of nonempty finite subsets of \(\Gamma\), and define
\[\phi_n(g)=\frac{|F_n\cap gF_n|}{|F_n|}.\]
Let \(R_n=T_{\phi_n^2}\).
Each \(\phi_n\) is real and positive definite,
since it is the coefficient of the left regular representation at \(|F_n|^{-1/2}\mathbf{1}_{F_n}\).
Moreover, \(\phi_n(e)=1\), its support is contained in \(F_nF_n^{-1}\), and \(\phi_n(g)\to1\) for every \(g\).
Thus \(R_n\) is finite rank and completely contractive.
It converges to the identity on the group algebra and, by density and uniform boundedness, on every \(L_p(\mathcal{M})\),
\(1\leq p<\infty\).

\begin{theorem}\label{thm:fourier}
    Let \(\Gamma\) be a countable amenable discrete group,
    and let \((F_n)\) be a left F\o lner sequence with \(\phi_n(g)=|F_n\cap gF_n|/|F_n|\).
    Let \(1<p<\infty\) and \(0<\delta<1\).
    There exists \(a=a(p,\delta)>1\) with the following properties.
    \begin{enumerate}
        \item If \(Z\) is a closed subspace of \(\mathsf{FM}_p(\Gamma)\) such that \(T_{\phi_n^2m}\in Z\) whenever \(T_m\in Z\) and \(n\geq1\), then
        \[S_Z\subseteq\bigcup_{j\geq1}B(aG_j,a-\delta),\]
        where \(G_j\in S_Z\) for every \(j\geq1\).
        \item If \(Z\) is a closed subspace of \(\mathsf{FM}^{\mathrm{cb}}_p(\Gamma)\) with the same invariance property, then,
            for every \(d\geq1\),
        \[S_{M_d(Z)}\subseteq\bigcup_{j\geq1}B(aG_{d,j},a-\delta),\]
        where \(\norm{G_{d,j}}_{M_d(Z)}=1\) for every \(j\geq1\).
    \end{enumerate}
    The covering directions in the first assertion, and the entries of the directions in the second assertion,
    are finitely supported Fourier multipliers.
    The parameter \(a\) is independent of \(\Gamma\), \((F_n)\), \(Z\), and \(d\).
\end{theorem}

\begin{proof}
    We choose \(0<\eta<1-\delta\) and \(0<t<\delta_p(\eta)/4\), depending only on \(p,\delta\).
    Let \(a=1/t\).
    For \(T\in Z\), we define \(Q_nT=R_nT\).
    Then \(Q_nT_m=T_{\phi_n^2m}\).
    Since \(\phi_n\) is supported in \(F_nF_n^{-1}\), we have \(\dim Q_nZ\leq |F_nF_n^{-1}|\).
    By Lemma~\ref{lem:fourier-factorization} and postcomposition, we have
    \[\norm{Q_n}\leq1\quad\text{and}\quad \norm{Q_n-tI_Z}\leq1-t(1-\eta)\]
    in the bounded case.
    The same estimates hold in the completely bounded norm in the second case.
    For \(T\in Z\) and \(x\in L_p(\mathcal{M})\), we have \(R_nTx\to Tx\).
    By testing on an almost norming vector, we obtain \(\liminf_n\norm{Q_nT}\geq\norm{T}\);
    complete contractivity supplies the reverse bound.
    We can therefore apply Lemma~\ref{lem:cover} in the bounded case.
    In the completely bounded case, we apply Lemma~\ref{lem:cb-cover} to the finite-rank maps \(R_n\).
    Since \(\delta<1-\eta\), we obtain the asserted coverings with \(a=1/t\) and radius \(a-\delta\).
    The centres belong to the finite-dimensional truncated images.
    For example, choosing \(\eta=(1-\delta)/2\) and \(t=\delta_p(\eta)/8\), we obtain
    \[a=\frac{8}{\delta_p((1-\delta)/2)}.\]
\end{proof}

For \(\Omega\subseteq\Gamma\), define
\[\mathsf{FM}_p(\Gamma;\Omega) =\{T_m\in\mathsf{FM}_p(\Gamma):m(g)=0\text{ for }g\notin\Omega\},\]
and define \(\mathsf{FM}^{\mathrm{cb}}_p(\Gamma;\Omega)\) in the same way.
The inequality \(|m(g)|\leq\norm{T_m}\) shows that these spaces are closed.
Multiplication by \(\phi_n^2\) preserves their prescribed support.

\begin{corollary}\label{cor:fourier-support}
    For every countable amenable group \(\Gamma\), every \(\Omega\subseteq\Gamma\), and every \(1<p<\infty\),
    the spaces \(\mathsf{FM}_p(\Gamma;\Omega)\) and \(\mathsf{FM}^{\mathrm{cb}}_p(\Gamma;\Omega)\) have the UBCP.
    For each \(0<\delta<1\), their coverings can be chosen as in Theorem~\ref{thm:fourier},
    with the same parameters at every matrix level in the completely bounded case.
    All centres have finite support contained in \(\Omega\).
\end{corollary}

If \(\Gamma=\bigcup_n\Gamma_n\) is locally finite, with finite subgroups \(\Gamma_n\subseteq\Gamma_{n+1}\), choose \(F_n=\Gamma_n\).
Then \(\phi_n=\mathbf{1}_{\Gamma_n}\) and \(R_n\) is the trace-preserving conditional expectation onto \(\operatorname{VN}(\Gamma_n)\).
Here the invariance condition in Theorem~\ref{thm:fourier} means that \(T_{m\mathbf{1}_{\Gamma_n}}\in Z\) whenever \(T_m\in Z\) and \(n\geq1\).

If \(g_0\ne e\), the rank-one map \(x\mapsto\tau(x)u_{g_0}\) sends \(1\) to \(u_{g_0}\) and hence is not a Fourier multiplier.

\subsection{Finite-rank approximations and free groups}
For free groups, we approximate positive definite radial smoothing maps in cb norm by finite-support multipliers.
The next lemma transfers the translation estimate to these approximations.
For related approximation results on free groups and reduced free products, see Ricard--Xu \cite{RicardXu2006}.

\begin{lemma}\label{lem:cb-perturbation}
    Let \(E,F\) be operator spaces and \(Z\subseteq\CB(E,F)\) a closed linear subspace.
    Suppose that \(R_n\in\CB(F,F)\) satisfy
    \[R_ny\to y\quad(y\in F),\qquad \norm{R_n}_{\mathrm{cb}}\leq1,\quad\text{and}\quad \norm{R_n-tI_F}_{\mathrm{cb}}\leq\mu<1,\]
    where \(0<t<1\).
    Assume that there are finite-rank maps \(A_n\) such that
    \[\varepsilon_n:=\norm{A_n-R_n}_{\mathrm{cb}}\to0, \qquad A_nZ\subseteq Z,\quad\text{and}\quad A_nZ\text{ is cb-norm separable}.\]
    For every \(0<\delta<(1-\mu)/t\) and every \(d\geq1\), there are \(G_{d,j}\in\bigcup_n S_{M_d(A_nZ)}\) such that
    \[S_{M_d(Z)}\subseteq\bigcup_j B(t^{-1}G_{d,j},t^{-1}-\delta).\]
    Their entries have finite rank.
    The same assertion holds at \(d=1\) for Banach spaces \(E,F\) and \(Z\subseteq\B(E,F)\) when all cb norms are replaced by operator norms.
\end{lemma}
\begin{proof}
    We define \(B_n=A_n/(1+\varepsilon_n)\).
    Then
    \begin{equation}\label{eq:perturbation-estimates}
        \norm{B_n}_{\mathrm{cb}}\leq1\quad\text{and}\quad \norm{B_n-R_n}_{\mathrm{cb}} \leq\frac{2\varepsilon_n}{1+\varepsilon_n}\leq2\varepsilon_n.
    \end{equation}
    By \eqref{eq:perturbation-estimates}, we have \(B_ny\to y\) and \(\norm{B_n-tI_F}_{\mathrm{cb}}\leq\mu+2\varepsilon_n\).
    We choose \(\mu<\mu'<1-t\delta\) and discard finitely many terms so that \(\mu+2\varepsilon_n\leq\mu'\).
    Since \(B_nZ=A_nZ\), we apply Lemma~\ref{lem:cb-cover} with \(M=1\) and \(\mu'\) in place of \(\mu\).
    The bounded version follows from Lemma~\ref{lem:cover} and the same postcomposition and norm detection argument.
\end{proof}

Let \(\mathbb{F}_r\) be the free group on \(r\) generators, \(2\leq r<\infty\), with standard word length \(|g|\).
Define
\[\Delta_k=T_{\mathbf{1}_{\{|g|=k\}}}\qquad(k\geq0).\]
The radial multiplier formula of Haagerup, Steenstrup and Szwarc \cite[Theorem~5.2]{HSS10} gives
\[\norm{\Delta_k:\VN(\mathbb{F}_r)\to\VN(\mathbb{F}_r)}_{\mathrm{cb}} \leq (k+1)+(k-1)_+\leq2(k+1).\]
Indeed, the associated Hankel matrix is \(H_{ij}=\mathbf{1}_{\{i+j=k\}}-\mathbf{1}_{\{i+j=k-2\}}\).
The two anti-diagonal matrices have trace norms \(k+1\) and \((k-1)_+\).
In the cited formula, the remaining factor is \((1-(2r-1)^{-1})(I-\sigma/(2r-1))^{-1}\) on the trace class,
where \(\sigma(A)=SAS^*\) and \(S\) is the unilateral shift.
Its norm is at most one.
Since \(\Delta_k\) is self-adjoint under the trace pairing, the same bound holds on \(L_1\).
Completely bounded interpolation then gives
\[\norm{\Delta_k:L_p(\VN(\mathbb{F}_r))\to L_p(\VN(\mathbb{F}_r))}_{\mathrm{cb}}\leq2(k+1) \qquad(1\leq p\leq\infty).\]

\begin{theorem}\label{thm:free-fourier}
    Let \(1<p<\infty\) and \(0<\delta<1\).
    Define
    \[a=\frac{8}{\delta_p((1-\delta)/2)}.\]
    For every \(2\leq r<\infty\) and every \(\Omega\subseteq\mathbb{F}_r\), the space \(\mathsf{FM}_p(\mathbb{F}_r;\Omega)\) admits a covering
    \[S_{\mathsf{FM}_p(\mathbb{F}_r;\Omega)} \subseteq\bigcup_jB(aG_j,a-\delta),\]
    where \(\norm{G_j}=1\) for every \(j\geq1\).
    For the completely bounded multiplier space, at every \(d\geq1\),
    \[S_{M_d(\mathsf{FM}^{\mathrm{cb}}_p(\mathbb{F}_r;\Omega))} \subseteq\bigcup_jB(aG_{d,j},a-\delta),\]
    where \(\norm{G_{d,j}}=1\) for every \(j\geq1\).
    The directions \(G_j\), and the entries of \(G_{d,j}\), are finitely supported Fourier multipliers supported in \(\Omega\).
    The covering parameters are independent of \(r\), \(\Omega\) and \(d\).
\end{theorem}
\begin{proof}
    Let \(\eta=(1-\delta)/2\) and \(t=\delta_p(\eta)/8\).
    By Haagerup's positive-definiteness lemma \cite[Lemma~1.2]{Haagerup79},
    the function \(\psi_s(g)=e^{-s|g|/2}\) is positive definite for every \(s>0\).
    We apply Lemma~\ref{lem:fourier-factorization} to \(R_s=T_{\psi_s^2}=T_{e^{-s|\cdot|}}\) and obtain
    \[\norm{R_s}_{\mathrm{cb}}\leq1\quad\text{and}\quad \norm{R_s-tI}_{\mathrm{cb}}\leq1-t(1-\eta)<1.\]
    We also have \(R_sx\to x\) as \(s\downarrow0\),
    first on the group algebra and then on \(L_p(\VN(\mathbb{F}_r))\) by density and contractivity.

    For \(s>0\) and \(N\geq0\), we define
    \[A_{s,N}=\sum_{k=0}^Ne^{-sk}\Delta_k.\]
    By the radial estimate, we obtain a cb-norm convergent expansion of \(R_s\) and
    \begin{equation}\label{eq:radial-tail}
        \norm{A_{s,N}-R_s}_{\mathrm{cb}} \leq2\sum_{k>N}(k+1)e^{-sk}\longrightarrow0.
    \end{equation}
    By \eqref{eq:radial-tail}, we choose \(N_n\geq n\) such that \(\norm{A_{1/n,N_n}-R_{1/n}}_{\mathrm{cb}}\leq1/n\).
    We define \(A_n=A_{1/n,N_n}\).
    Since balls in \(\mathbb{F}_r\) are finite, \(A_n\) has finite rank.
    Postcomposition by \(A_n\) preserves each prescribed-support multiplier space and has finite-dimensional image,
    of dimension at most \(|\Omega\cap\{g:|g|\leq N_n\}|\).
    Since \(\delta<1-\eta\), we obtain both coverings from Lemma~\ref{lem:cb-perturbation} and its bounded version,
    with \(\mu=1-t(1-\eta)\) and \(a=1/t\).
\end{proof}

Both full multiplier spaces in this theorem are nonseparable.
Indeed, if \(a_1,\ldots,a_r\) are the generators, let \(\chi_\theta(a_1)=e^{i\theta}\) and \(\chi_\theta(a_j)=1\) for \(j>1\),
where \(0\leq\theta<2\pi\).
Each \(T_{\chi_\theta}\) is induced by a trace-preserving \(*\)-automorphism, and has operator and cb norm one.
For \(\theta\ne\varphi\), testing on \(u_{a_1^k}\) gives
\[\norm{T_{\chi_\theta}-T_{\chi_\varphi}} \geq\sup_{k\in\mathbb{Z}}|e^{ik\theta}-e^{ik\varphi}|\geq1.\]

The same argument applies to an integer-valued length \(\ell\) on a countable group whose balls are finite,
for which \(e^{-s\ell}\) is positive definite for every \(s>0\) and \(\norm{T_{\mathbf{1}_{\{\ell=k\}}}}_{\mathrm{cb}} \leq C(1+k)^\beta\) on the group von Neumann algebra,
for some \(C>0\), \(\beta\geq0\).
The polynomial bound makes the exponential tails summable;
\(C,\beta\) affect the truncation orders but not the covering parameters.

\begin{corollary}\label{cor:fourier-quotients}
    Let \(\Gamma\) be a countable amenable group or a finitely generated free group, let \(\Omega\subseteq\Gamma\), and let \(1<p<\infty\).
    Equip \(Z=\mathsf{FM}_p(\Gamma;\Omega)\), respectively \(Z=\mathsf{FM}^{\mathrm{cb}}_p(\Gamma;\Omega)\), with its natural dual structure.
    For every weak-star closed linear subspace \(J\subseteq Z\) and every \(0<\delta<1\),
    the quotient \(Z/J\) has a covering with centre norm \(a=8/\delta_p((1-\delta)/2)\) and radius \(a-\delta\).
    In the completely bounded case the same parameters work on every \(M_d(Z/J)\).
    The directions have representatives consisting of finitely supported Fourier multipliers supported in \(\Omega\).
\end{corollary}
\begin{proof}
    Let \(F=L_p(\VN(\Gamma))\).
    Since \(F\) is reflexive, we identify \(\B(F)\) with \((F\widehat{\otimes}_\pi F^*)^*\) and \(\CB(F,F)\) with \((F\widehat{\otimes} F^*)^*\),
    where the latter tensor product is the operator space projective tensor product;
    see \cite{ER00}.
    The relations \(T(u_g)\in\mathbb{C} u_g\) for \(g\in\Gamma\) and \(T(u_g)=0\) for \(g\notin\Omega\) are weak-star closed.
    Thus \(Z\) is a dual space with the predual induced by the corresponding tensor product.

    In the amenable case, we let \(B_n=T_{\phi_n^2}\), as in Theorem~\ref{thm:fourier}.
    In the free group case, with the notation of Theorem~\ref{thm:free-fourier}, we let \(B_n=A_n/(1+1/n)\).
    In both cases \(B_n\) is a finite-support complete contraction and \(B_nx\to x\) for \(x\in F\).
    We define \(Q_nT=B_nT\).
    By pointwise norm convergence and uniform boundedness, we have \(Q_nT\to T\) weak-star on the stated predual.
    Moreover, \(Q_nZ\) is finite dimensional.
    With \(t=\delta_p((1-\delta)/2)/8\), after discarding finitely many terms we have
    \[\norm{Q_n}\leq1\quad\text{and}\quad \norm{Q_n-tI_Z}\leq\mu'<1-t\delta.\]
    The same bounds hold in the cb norm in the second case.
    We obtain the assertion by applying Theorem~\ref{thm:quotient} and Proposition~\ref{prop:cb-quotient}.
    Their directions lie in the quotient images of \(Q_nZ\), which proves the statement about representatives.
\end{proof}

For a concrete instance of Corollary~\ref{cor:fourier-quotients}, let \(W=\mathsf{FM}^{\mathrm{cb}}_4(\mathbb{F}_2)/\C I\).
The one-dimensional subspace \(\C I\) is weak-star closed.
The finite-support maps \(B_n\) used above satisfy \(B_nI=B_n\notin\C I\).
For every \(0<\delta<1\), the corollary gives
\[S_{M_d(W)}\subseteq\bigcup_{j\geq1}B(aG_{d,j},a-\delta)\qquad(d\geq1),\]
where \(a=8/\delta_4((1-\delta)/2)\) and \(\norm{G_{d,j}}=1\).
Each \(G_{d,j}\) has a matrix representative consisting of finitely supported Fourier multipliers.
The sufficient conditions in \cite[Theorems~3.6--3.7]{LMS26} use predual geometry to obtain the UBCP,
while \cite[Theorem~3.5]{MPSS26} gives sufficient conditions for the BCP of operator subspaces containing all finite-rank operators.

\subsection{An endpoint example}
We give a locally finite noncommutative group whose completely bounded Fourier multiplier spaces fail the BCP at \(p=1,\infty\).

Here \(C^*(\Gamma)\) is the full group \(C^*\)-algebra and \(B(\Gamma)=C^*(\Gamma)^*\) is its Fourier--Stieltjes algebra.

\begin{lemma}\label{lem:fourier-endpoint-identification}
    For a countable locally finite group \(\Gamma\), the correspondence \(m\mapsto T_m\) gives a linear isometry
    \[B(\Gamma)\cong\mathsf{FM}^{\mathrm{cb}}_\infty(\Gamma).\]
    Moreover, \(\mathsf{FM}^{\mathrm{cb}}_1(\Gamma)\) is linearly isometric to \(\mathsf{FM}^{\mathrm{cb}}_\infty(\Gamma)\).
\end{lemma}

\begin{proof}
    Since \(\Gamma\) is amenable, we obtain the first assertion from the isometric identification of its Fourier--Stieltjes algebra with the completely bounded multipliers of its Fourier algebra \cite[Corollary~4.3]{Spronk}.
    Under the trace pairing \(\langle x,y\rangle=\tau(xy)\),
    the preadjoint of \(T_m\) is the multiplier with symbol \(m^\vee(g)=m(g^{-1})\) on \(L_1(\mathcal{M})\).
    Conversely, the adjoint of a multiplier on \(L_1\) is normal on \(\mathcal{M}\) with the inverted symbol.
    Since adjoint passage preserves the cb norm, we obtain the second assertion.
\end{proof}

\begin{proposition}\label{prop:fourier-endpoints}
    Let \(A=\bigoplus_{n\geq1}\mathbb{Z}_2\) and \(\Gamma_0=A\times S_3\).
    For \(1<p<\infty\), the spaces \(M_d(\mathsf{FM}^{\mathrm{cb}}_p(\Gamma_0))\), \(d\geq1\),
    have common UBCP covering constants depending only on \(p\).
    At \(p=1\) and \(p=\infty\), \(\mathsf{FM}^{\mathrm{cb}}_p(\Gamma_0)\) fails the BCP.
\end{proposition}

\begin{proof}
    By Theorem~\ref{thm:fourier}, we obtain the assertion for \(1<p<\infty\).
    Let \(K=\widehat{A}=\{-1,1\}^{\mathbb{N}}\).
    Since \(C^*(S_3)\cong\mathbb{C}\oplus_\infty\mathbb{C}\oplus_\infty M_2\), we have
    \[C^*(\Gamma_0) \cong C(K)\oplus_\infty C(K)\oplus_\infty C(K;M_2) =C(K)\oplus_\infty D.\]
    By Lemma~\ref{lem:fourier-endpoint-identification}, we have the isometric identification
    \[\mathsf{FM}^{\mathrm{cb}}_\infty(\Gamma_0) \cong M(K)\oplus_1D^*.\]
    Here \(M(K)=C(K)^*\) has the total variation norm.
    The failure of its BCP is contained in \cite[Corollary~20]{GLM19}.
    Given countably many \(f_j=(\mu_j,\nu_j)\), we choose \(x\in K\) which is not an atom of any \(|\mu_j|\).
    Then \(f=(\delta_x,0)\) has norm one and
    \[\|f-f_j\| =\|\delta_x-\mu_j\|+\|\nu_j\| =1+\|\mu_j\|+\|\nu_j\| =1+\|f_j\|.\]
    Thus the space fails the BCP at \(p=\infty\),
    and the preadjoint identification in Lemma~\ref{lem:fourier-endpoint-identification} gives failure at \(p=1\).
\end{proof}

\subsection{Compatible restrictions}
For locally finite groups, restriction to finite subgroups gives the following bounded inverse-limit representation.

Now let \(\Gamma=\bigcup_n\Gamma_n\) be a countable locally finite group, where \(\Gamma_n\) are increasing finite subgroups.
Let \(E_n:L_p(\operatorname{VN}(\Gamma))\to L_p(\operatorname{VN}(\Gamma_n))\) denote the trace-preserving conditional expectation,
and define
\[Q_nT_m=E_nT_m=T_{m\mathbf{1}_{\Gamma_n}}.\]
For Banach spaces \(V_n\) with contractive restriction maps \(\rho_{n+1,n}\), define
\[\varprojlim\nolimits_b V_n =\left\{(v_n):\rho_{n+1,n}v_{n+1}=v_n, \ \sup_n\|v_n\|<\infty\right\},\]
and equip this space with the norm \(\|(v_n)\|=\sup_n\|v_n\|\).
For operator spaces, equip this space with the corresponding supremum matrix norms.

\begin{proposition}\label{prop:fourier-inverse-limit}
    For \(1<p<\infty\), restriction of symbols induces an isometry
    \[\mathsf{FM}_p(\Gamma) \cong\varprojlim\nolimits_b\mathsf{FM}_p(\Gamma_n)\]
    and a complete isometry
    \[\mathsf{FM}^{\mathrm{cb}}_p(\Gamma) \cong\varprojlim\nolimits_b\mathsf{FM}^{\mathrm{cb}}_p(\Gamma_n).\]
    Extension of a symbol on \(\Gamma_n\) by zero is an isometry in the first case and a complete isometry in the second case.
\end{proposition}

\begin{proof}
    We first consider restriction and extension by zero.
    The inclusion \(J_n:L_p(\operatorname{VN}(\Gamma_n))\to L_p(\mathcal{M})\) is a complete isometry and \(E_nJ_n=I\).
    Hence restriction of a multiplier is contractive, and completely contractive in the completely bounded setting.
    Extension by zero is the map \(T\mapsto J_nTE_n\);
    restriction is a left inverse, so extension has the asserted isometric properties.

    For a compatible family \((m_n)\) with \(\sup_n\norm{T_{m_n}}<\infty\), we define a symbol \(m\) on \(\Gamma\) by \(m|_{\Gamma_n}=m_n\).
    Every element of \(\mathbb{C}\Gamma\) belongs to some \(\mathbb{C}\Gamma_N\),
    where the associated multiplier has norm at most \(\sup_n\norm{T_{m_n}}\).
    By density, we extend this multiplier to \(L_p(\mathcal{M})\) with the same bound.
    For a matrix of compatible symbols, we fix its size \(d\) and an amplification size \(k\).
    Every element of \(M_k(\mathbb{C}\Gamma)\) belongs to \(M_k(\mathbb{C}\Gamma_N)\) for some \(N\).
    The finite-stage matrix norm bounds therefore give the same bound for the map into \(M_k(M_d(L_p(\mathcal{M})))\).
    Using density at each fixed \(k\) and then taking the supremum over \(k\), we obtain complete boundedness with the prescribed norm.
    We obtain the reverse inequalities by restriction, at every \(d\).
\end{proof}

In particular, for either multiplier norm, the sequence \((\norm{Q_nT})\) increases to \(\norm{T}\).
If \(\Gamma\) is infinite, however,
the identity multiplier satisfies \(\|T_{\mathbf{1}}-Q_nT_{\mathbf{1}}\|\geq1,\) as is seen by testing on \(u_g\) with \(g\notin\Gamma_n\).

\section*{Acknowledgments}
Rui Liu and Jie Shen were partially supported by the National Natural Science Foundation of China (Grant No. 12471131).

\section*{Declaration of AI Use}
We used GPT 6 Astra (OpenAI) for language editing and assistance with mathematical correctness checks.
We also used it to help assess the novelty and scope of our results in relation to existing work, including comparative examples,
elementary deductions from related results, and examination of the applicability and limitations of the proof methods in the cited papers.
The tool also suggested examples from noncommutative analysis.
It was not used to develop the core theorems or determine the structure of the manuscript.
The authors take full responsibility for the mathematical content, interpretation of the literature, and final text.

\bibliographystyle{abbrv}
\bibliography{cite}

\end{document}

%% file: abbreviation.tex
\newcommand{\B}{\mathcal{B}}
\newcommand{\F}{\mathcal{F}}
\newcommand{\CB}{\operatorname{CB}}
\newcommand{\Id}{I}
\newcommand{\N}{\mathbb{N}}
\newcommand{\C}{\mathbb{C}}

\newcommand{\M}{\mathcal{M}}
\newcommand{\norm}[1]{\left\lVert#1\right\rVert}

\DeclareMathOperator{\VN}{VN}

%% file: cite.bib
@article{BLS25,
  author  = {Bao, Q. and Liu, R. and Shen, J.},
  title   = {The ball-covering property of non-commutative spaces of operators on {Banach} spaces},
  journal = {Banach J. Math. Anal.},
  volume  = {19},
  number  = {3},
  pages   = {34},
  year    = {2025}
}

@article{BLS27,
  author  = {Bao, Q. and Liu, R. and Shen, J.},
  title   = {Approximation properties and quantitative estimation for uniform ball-covering property of operator spaces},
  journal = {J. Funct. Anal.},
  volume  = {292},
  number  = {2},
  pages   = {111657},
  year    = {2027}
}

@article{CaspersSalle,
  author  = {Caspers, M. and de la Salle, M.},
  title   = {{Schur} and {Fourier} multipliers of an amenable group acting on non-commutative {$L^p$}-spaces},
  journal = {Trans. Amer. Math. Soc.},
  volume  = {367},
  number  = {10},
  pages   = {6997--7013},
  year    = {2015}
}

@article{Cheng06,
  author  = {Cheng, L.},
  title   = {Ball-covering property of {Banach} spaces},
  journal = {Israel J. Math.},
  volume  = {156},
  pages   = {111--123},
  year    = {2006}
}

@article{CCS09,
  author  = {Cheng, L. and Cheng, Q. and Shi, H.},
  title   = {Minimal ball-coverings in {Banach} spaces and their application},
  journal = {Studia Math.},
  volume  = {192},
  number  = {1},
  pages   = {15--27},
  year    = {2009}
}

@article{Clarkson,
  author  = {Clarkson, J. A.},
  title   = {Uniformly convex spaces},
  journal = {Trans. Amer. Math. Soc.},
  volume  = {40},
  pages   = {396--414},
  year    = {1936}
}

@book{ER00,
  author    = {Effros, E. G. and Ruan, Z.-J.},
  title     = {Operator Spaces},
  publisher = {Clarendon Press},
  address   = {Oxford},
  year      = {2000}
}

@article{GLM19,
  author  = {Guirao, A. J. and Lissitsin, A. and Montesinos, V.},
  title   = {Some remarks on the ball-covering property},
  journal = {J. Math. Anal. Appl.},
  volume  = {479},
  number  = {1},
  pages   = {608--620},
  year    = {2019}
}

@article{Haagerup79,
  author  = {Haagerup, U.},
  title   = {An example of a non nuclear {$C^*$}-algebra, which has the metric approximation property},
  journal = {Invent. Math.},
  volume  = {50},
  number  = {3},
  pages   = {279--293},
  year    = {1979}
}

@article{HaagerupMusat2011,
  author  = {Haagerup, U. and Musat, M.},
  title   = {Factorization and dilation problems for completely positive maps on von {Neumann} algebras},
  journal = {Comm. Math. Phys.},
  volume  = {303},
  pages   = {555--594},
  year    = {2011}
}

@article{HSS10,
  author  = {Haagerup, U. and Steenstrup, T. and Szwarc, R.},
  title   = {{Schur} multipliers and spherical functions on homogeneous trees},
  journal = {Int. J. Math.},
  volume  = {21},
  number  = {10},
  pages   = {1337--1382},
  year    = {2010}
}

@article{HKL24,
  author  = {Huang, J. and Kudaybergenov, K. and Liu, R.},
  title   = {The ball-covering property of noncommutative symmetric spaces},
  journal = {Studia Math.},
  volume  = {277},
  number  = {2},
  pages   = {169--190},
  year    = {2024}
}

@misc{LMS26,
  author = {Langemets, J. and M{\~o}ttus, E. and Saealle, N.},
  title  = {Dual {Banach} spaces with the ball-covering property},
  note   = {arXiv:\allowbreak 2607.10409v1},
  year   = {2026}
}

@book{LT,
  author    = {Lindenstrauss, J. and Tzafriri, L.},
  title     = {Classical {Banach} Spaces {II}: Function Spaces},
  publisher = {Springer-Verlag},
  address   = {Berlin--New York},
  year      = {1979}
}

@article{LLLZ22,
  author  = {Liu, M. and Liu, R. and Lu, J. and Zheng, B.},
  title   = {Ball covering property from commutative function spaces to non-commutative spaces of operators},
  journal = {J. Funct. Anal.},
  volume  = {283},
  number  = {1},
  pages   = {109502},
  year    = {2022}
}

@article{LZ21,
  author  = {Luo, Z. and Zheng, B.},
  title   = {The strong and uniform ball-covering properties},
  journal = {J. Math. Anal. Appl.},
  volume  = {499},
  number  = {1},
  pages   = {125034},
  year    = {2021}
}

@misc{MPSS26,
  author = {Mishra, A. and Paul, K. and Sain, D. and Sohel, S.},
  title  = {On the ball covering property of {$\mathbb L(\mathbb X,\mathbb Y)$}},
  note   = {arXiv:\allowbreak 2607.10353v2},
  year   = {2026}
}

@article{NeuwirthRicard2011,
  author  = {Neuwirth, S. and Ricard, {\'E}.},
  title   = {Transfer of {Fourier} multipliers into {Schur} multipliers and sumsets in a discrete group},
  journal = {Canad. J. Math.},
  volume  = {63},
  number  = {5},
  pages   = {1161--1187},
  year    = {2011}
}

@article{Pisier98,
  author  = {Pisier, G.},
  title   = {Non-commutative vector valued {$L_p$}-spaces and completely {$p$}-summing maps},
  journal = {Ast\'erisque},
  volume  = {247},
  note    = {vi+131 pp.},
  year    = {1998}
}

@book{Pisier2003,
  author    = {Pisier, G.},
  title     = {Introduction to Operator Space Theory},
  series    = {London Mathematical Society Lecture Note Series},
  volume    = {294},
  publisher = {Cambridge University Press},
  address   = {Cambridge},
  year      = {2003}
}

@incollection{PisierXu03,
  author    = {Pisier, G. and Xu, Q.},
  title     = {Non-commutative {$L_p$}-spaces},
  editor    = {Johnson, W. B. and Lindenstrauss, J.},
  booktitle = {Handbook of the Geometry of {Banach} Spaces},
  volume    = {2},
  publisher = {North-Holland},
  address   = {Amsterdam},
  pages     = {1459--1517},
  year      = {2003}
}

@article{Ricard08,
  author  = {Ricard, {\'E}.},
  title   = {A {Markov} dilation for self-adjoint {Schur} multipliers},
  journal = {Proc. Amer. Math. Soc.},
  volume  = {136},
  number  = {12},
  pages   = {4365--4372},
  year    = {2008}
}

@article{RicardXu2006,
  author  = {Ricard, {\'E}. and Xu, Q.},
  title   = {{Khintchine} type inequalities for reduced free products and applications},
  journal = {J. Reine Angew. Math.},
  volume  = {599},
  pages   = {27--59},
  year    = {2006}
}

@article{Schach83,
  author  = {Schachermayer, W.},
  title   = {Norm attaining operators and renormings of {Banach} spaces},
  journal = {Israel J. Math.},
  volume  = {44},
  number  = {3},
  pages   = {201--212},
  year    = {1983}
}

@article{Schechtman,
  author  = {Schechtman, G.},
  title   = {Almost isometric {$L_p$} subspaces of {$L_p(0,1)$}},
  journal = {J. London Math. Soc. (2)},
  volume  = {20},
  number  = {3},
  pages   = {516--528},
  year    = {1979}
}

@article{SC13,
  author  = {Shang, S. and Cui, Y.},
  title   = {Ball-covering property in uniformly non-{$\ell_3^{(1)}$} {Banach} spaces and application},
  journal = {Abstr. Appl. Anal.},
  volume  = {2013},
  pages   = {873943},
  year    = {2013}
}

@article{SZ26,
  author  = {Siju, S. and Zheng, B.},
  title   = {Ball covering property on operators and {Calkin} algebras},
  journal = {Studia Math.},
  volume  = {290},
  number  = {3},
  pages   = {309--320},
  year    = {2026}
}

@article{Spronk,
  author  = {Spronk, N.},
  title   = {Measurable {Schur} multipliers and completely bounded multipliers of the {Fourier} algebras},
  journal = {Proc. London Math. Soc. (3)},
  volume  = {89},
  number  = {1},
  pages   = {161--192},
  year    = {2004}
}
